\documentclass[a4paper,11pt,leqno]{smfart}

\usepackage{amssymb,amsmath,enumerate,verbatim,mathrsfs,graphics,graphicx,mathptm,float,mathtools}
\usepackage[T1]{fontenc}
\usepackage[english, francais]{babel}
\usepackage[all]{xy}
\usepackage[pdfstartview=FitH,
            colorlinks=true,
            linkcolor=blue,
            urlcolor=blue,
            citecolor=red,
            bookmarks,
            bookmarksopen=true,
            bookmarksnumbered=true]{hyperref}

\usepackage{cleveref}
\theoremstyle{plain}

\newtheorem{theoalph}{Theorem}

\newtheorem{thmalph}[theoalph]{Theorem}

\newtheorem{coralph}[theoalph]{Corollary}

\theoremstyle{definition}

\theoremstyle{remark}

\theoremstyle{plain}
\newtheorem{thmsec}{Theorem}[section]

\newtheorem{pro}[thmsec]{Proposition}
\newtheorem{lem}[thmsec]{Lemma}
\newtheorem{cor}[thmsec]{Corollary}
\newtheorem{prob}{Problem}

\theoremstyle{definition}
\newtheorem{defin}[thmsec]{Definition}

\theoremstyle{remark}
\newtheorem{rem}[thmsec]{Remark}
\newtheorem{rems}[thmsec]{Remarks}
\newtheorem{eg}[thmsec]{Example}

\def\og{\leavevmode\raise.3ex\hbox{$\scriptscriptstyle\langle\!\langle$~}}
\def\fg{\leavevmode\raise.3ex\hbox{~$\!\scriptscriptstyle\,\rangle\!\rangle$}}

\addtocounter{section}{0}             
\numberwithin{equation}{section}       

\newcommand{\Z}{\mathbb{Z}}

\newcommand{\C}{\mathbb{C}}

\newcommand{\pp}{\mathbb{P}^{2}_{\mathbb{C}}}
\newcommand{\pd}{\mathbb{\check{P}}^{2}_{\mathbb{C}}}

\newcommand\X{\mathrm{X}}

\newcommand\Sing{\mathrm{Sing}}
\newcommand\Tang{\mathrm{Tang}}
\newcommand\Leg{\mathrm{Leg}}
\newcommand\IF{\mathrm{I}_{\mathcal{F}}}
\newcommand\IinvF{\mathrm{I}_{\mathcal{F}}^{\mathrm{inv}}}
\newcommand\ItrF{\mathrm{I}_{\mathcal{F}}^{\hspace{0.2mm}\mathrm{tr}}}

\newcommand\F{\mathcal{F}}
\newcommand\W{\mathcal{W}}

\newcommand\omegaoverline{{\mspace{2mu}\overline{\mspace{-1.4mu}\omega\mspace{-1.4mu}}\mspace{2mu}}}
\newcommand\Omegaoverline{{\mspace{2mu}\overline{\mspace{-1.4mu}\Omega\mspace{-1.4mu}}\mspace{2mu}}}
\newcommand\thetaoverline{{\mspace{2mu}\overline{\mspace{-1.4mu}\theta\mspace{-1.4mu}}\mspace{2mu}}}

\newcommand*{\transp}[2][-3mu]{\ensuremath{\mskip1mu\prescript{\smash{\mathrm t\mkern#1}}{}{\mathstrut#2}}}

\begin{document}
\selectlanguage{english}

\title[Classification of foliations of degree three on $\pp$ with a flat Legendre transform]{Classification of foliations of degree three on $\pp$ with a flat Legendre transform}
\date{\today}

\author{Samir \textsc{Bedrouni}}

\address{Facult\'e de Math\'ematiques, USTHB, BP $32$, El-Alia, $16111$ Bab-Ezzouar, Alger, Alg\'erie}
\email{sbedrouni@usthb.dz}

\author{David \textsc{Mar\'{\i}n}}

\thanks{D. Mar\'{\i}n acknowledges financial support from the Spanish Ministry of Economy and Competitiveness, through grant MTM2015-66165-P and the "Mar\'{\i}a de Maeztu" Programme for Units of Excellence in R\&D (MDM-2014-0445).}

\address{Departament de Matem\`{a}tiques Universitat Aut\`{o}noma de Barcelona E-08193  Bellaterra (Barcelona) Spain}

\email{davidmp@mat.uab.es}





\maketitle{}

\begin{abstract}
The set $\mathbf{F}(3)$ of foliations of degree three on the complex projective plane can be identified with a Zariski's open set of a projective space of dimension $23$ on which acts $\mathrm{Aut}(\mathbb{P}^{2}_{\mathbb{C}})$. The subset $\mathbf{FP}(3)$ of $\mathbf{F}(3)$ consisting of foliations of $\mathbf{F}(3)$ with a flat Legendre transform (dual web) is a Zariski closed subset of $\mathbf{F}(3)$. We classify up to automorphism of $\mathbb{P}^{2}_{\mathbb{C}}$ the elements of $\mathbf{FP}(3)$. More precisely, we show that up to automorphism there are $16$ foliations of degree three with a flat Legendre transform. From this classification we deduce that $\mathbf{FP}(3)$ has exactly~$12$ irreducible components. We also deduce that up to automorphism there are $4$ convex foliations of degree three on~$\mathbb{P}^2.$
\noindent{\it 2010~Mathematics Subject Classification. --- 14C21, 32S65, 53A60.}
\end{abstract}

\section*{Introduction}
\bigskip

\noindent  A (regular) $d$-web on $(\mathbb{C}^2,0)$ is the data of a family $\{\F_1,\F_2,\ldots,\F_d\}$ of regular holomorphic foliations on $(\mathbb{C}^2,0)$ which are pairwise transverse at the origin. The first significant result in the study of webs was obtained by W. \textsc{Blaschke} and J. \textsc{Dubourdieu} around $1920$. 
They have introduced, for any regular $3$-web $\mathcal{W}$ on $(\mathbb{C}^2,0),$ a differential $2$-form $K(\mathcal{W})$ known as the \textsc{Blaschke} curvature of $\mathcal{W}$, whose vanishing implies (\cite{BD28}) that $\mathcal{W}$ is analytically equivalent to the trivial $3$-web defined by $\mathrm{d}x.\mathrm{d}y.\mathrm{d}(x+y)$.
The curvature of a $d$-web~$\W$ with $d>3$ is defined as the sum of the \textsc{Blaschke} curvatures of all $3$-subwebs of $\W$. A web with zero curvature is called flat. This notion is useful for the classification of maximal rank webs. Indeed, a result of N. \textsc{Mih\u{a}ileanu} shows that the flatness is a necessary condition for the maximality of the rank, \emph{see} for instance  \cite{Hen06,Rip07}.

\noindent Recently, the study of global holomorphic webs defined on complex surfaces has been updated, \emph{see} for instance \cite{CL07,PP09,MP13}. In the sequel we will focus on webs of the complex projective plane. A (global) $d$-web on $\pp$ is given in an affine chart $(x,y)$ by an implicit differential equation $F(x,y,y')=0$, where $F(x,y,p)=\sum_{i=0}^{d}a_{i}(x,y)p^{d-i}\in\mathbb{C}[x,y,p]$ is a reduced polynomial whose coefficient $a_0$ is not identically zero. In a neighborhood of a point $z_0=(x_0,y_0)$ such that  $a_{0}(x_{0},y_{0})\Delta(x_{0},y_{0})\neq 0$, being $\Delta(x,y)$ the $p$-discriminant of $F$, the integral curves of this equation define a regular $d$-web on $(\C^2,z_0)$.

\noindent The curvature of a web $\W$ on $\pp$ is a meromorphic $2$-form with poles along the discriminant $\Delta(\W)$, \emph{see} \S\ref{subsec:courbure-platitude}.

\noindent D.~\textsc{Mar\'{\i}n} and J.V. \textsc{Pereira} have shown, in \cite{MP13}, how to associate to every degree $d$ foliation $\F$ on $\pp$, a global $d$-web on the dual projective plane $\pd$, called {\sl \textsc{Legendre} transform} of $\F$ and denoted by $\Leg\F$. The leaves of $\Leg\F$ are the dual curves of the leaves of $\F$,  \emph{see}  \S\ref{subsec-tissus}.

\noindent The set  $\mathbf{F}(d)$ of degree $d$ foliation on $\pp$ can be naturally identified with a \textsc{Zariski} open subset of the projective space  $\mathbb{P}_{\C}^{(d+2)^{2}-2}$. The automorphism group of $\pp$ acts on $\mathbf{F}(d)$; the orbit of an element $\F\in\mathbf{F}(d)$ under the action of $\mathrm{Aut}(\pp)=\mathrm{PGL}_3(\mathbb{C})$ will be denoted by  $\mathcal{O}(\F)$. The subset $\mathbf{FP}(d)$ of $\mathbf{F}(d)$ consisting of $\F\in\mathbf{F}(d)$ such that $\Leg\F$ is flat is \textsc{Zariski} closed in $\mathbf{F}(d)$ and saturated by the action of $\mathrm{Aut}(\pp)$.

\noindent In \cite{MP13} the authors pose a problem concerning the geometry of webs on $\pp$ which, in the framework of the foliations on $\pp,$ consists in the description of certain irreducible components of $\mathbf{FP}(d)$. The first nontrivial case that we encounter is the one where $d=3.$ In this paper we describe the decomposition of $\mathbf{FP}(3)$ into its irreducible components. In order to do this, we begin by establishing the classification, up to isomorphism, of the foliations of $\mathbf{FP}(3).$ In a previous work \cite{BM17}, we have shown (\cite[Theorem~5.1]{BM17}) that up to isomorphism there are eleven \textsl{homogeneous} foliations ({\it i.e.} invariant by homotheties) of degree $3$, denoted $\mathcal{H}_{1},\ldots,\mathcal{H}_{11}$, with a flat \textsc{Legendre} transform. On the other hand, we have also proved (\cite[Theorem~6.1]{BM17}) that if a foliation of $\mathbf{FP}(3)$ has only non-degenerate singularities ({\it i.e.} singularities with \textsc{Milnor} number $1$), then it is linearly conjugated to the \textsc{Fermat} foliation $\F_3$ defined by the $1$-form  $(x^{3}-x)\mathrm{d}y-(y^{3}-y)\mathrm{d}x.$ In \S\ref{sec:cas-dégénéré} by studying the flatness of the dual web of a foliation $\F\in\mathbf{F}(3)$ having at least one degenerate singularity, we obtain the classification, up to automorphism of $\pp$, of the elements of $\mathbf{FP}(3).$
\begin{thmalph}\label{thmalph:classification}
{\sl Up to automorphism of $\pp$ there are sixteen foliations of degree three $\mathcal{H}_1,\ldots,\mathcal{H}_{11},\mathcal{F}_1,\ldots,\mathcal{F}_{5}$ on the complex projective plane having a flat \textsc{Legendre} transform. They are respectively described in affine chart by the following $1$-forms
\begin{itemize}
\item [\texttt{1. }]  \hspace{1mm} $\omega_1\hspace{1mm}=y^3\mathrm{d}x-x^3\mathrm{d}y$;
\item [\texttt{2. }]  \hspace{1mm} $\omega_2\hspace{1mm}=x^3\mathrm{d}x-y^3\mathrm{d}y$;
\item [\texttt{3. }]  \hspace{1mm} $\omega_3\hspace{1mm}=y^2(3x+y)\mathrm{d}x-x^2(x+3y)\mathrm{d}y$;
\item [\texttt{4. }]\hspace{1mm}   $\omega_4\hspace{1mm}=y^2(3x+y)\mathrm{d}x+x^2(x+3y)\mathrm{d}y$;
\item [\texttt{5. }]\hspace{1mm} $\omega_{5}\hspace{1mm}=2y^3\mathrm{d}x+x^2(3y-2x)\mathrm{d}y$;
\item [\texttt{6. }]\hspace{1mm} $\omega_{6}\hspace{1mm}=(4x^3-6x^2y+4y^3)\mathrm{d}x+x^2(3y-2x)\mathrm{d}y$;
\item [\texttt{7. }]\hspace{1mm} $\omega_{7}\hspace{1mm}=y^3\mathrm{d}x+x(3y^2-x^2)\mathrm{d}y$;
\item [\texttt{8. }]\hspace{1mm} $\omega_{8}\hspace{1mm}=x(x^2-3y^2)\mathrm{d}x-4y^3\mathrm{d}y$;
\item [\texttt{9. }]\hspace{1mm} $\omega_{9}\hspace{1mm}=y^{2}\left((-3+\mathrm{i}\sqrt{3})x+2y\right)\mathrm{d}x+
                                                       x^{2}\left((1+\mathrm{i}\sqrt{3})x-2\mathrm{i}\sqrt{3}y\right)\mathrm{d}y$;
\item [\texttt{10. }]\hspace{-1mm} $\omega_{10}=(3x+\sqrt{3}y)y^2\mathrm{d}x+(3y-\sqrt{3}x)x^2\mathrm{d}y$;
\item [\texttt{11. }]\hspace{-1mm} $\omega_{11}=(3x^3+3\sqrt{3}x^2y+3xy^2+\sqrt{3}y^3)\mathrm{d}x+(\sqrt{3}x^3+3x^2y+3\sqrt{3}xy^2+3y^3)\mathrm{d}y$;
\item [\texttt{12. }]\hspace{-1mm} $\omegaoverline_{1}\hspace{1mm}=y^{3}\mathrm{d}x+x^{3}(x\mathrm{d}y-y\mathrm{d}x)$;
\item [\texttt{13. }]\hspace{-1mm} $\omegaoverline_{2}\hspace{1mm}=x^{3}\mathrm{d}x+y^{3}(x\mathrm{d}y-y\mathrm{d}x)$;
\item [\texttt{14. }]\hspace{-1mm} $\omegaoverline_{3}\hspace{1mm}=(x^{3}-x)\mathrm{d}y-(y^{3}-y)\mathrm{d}x$;
\item [\texttt{15. }]\hspace{-1mm} $\omegaoverline_{4}\hspace{1mm}=(x^{3}+y^{3})\mathrm{d}x+x^{3}(x\mathrm{d}y-y\mathrm{d}x)$;
\item [\texttt{16. }]\hspace{-1mm} $\omegaoverline_{5}\hspace{1mm}=y^{2}(y\mathrm{d}x+2x\mathrm{d}y)+x^{3}(x\mathrm{d}y-y\mathrm{d}x)$.
\end{itemize}
}
\end{thmalph}

\noindent The orbits of $\F_1$ and $\F_2$ are both of dimension $6$ which is the minimal dimension possible, and this in any degree greater than or equal to $2$ (\cite[Proposition~2.3]{CDGBM10}). D.~\textsc{Cerveau}, J.~\textsc{D\'eserti}, D.~\textsc{Garba Belko} and R.~\textsc{Meziani} have shown that in degree $2$ there are exactly two orbits of dimension $6$ (\cite[Proposition 2.7]{CDGBM10}). Theorem~\ref{thmalph:classification} allows us to establish a similar result in degree $3$:
\begin{coralph}\label{coralph:dim-min}
{\sl Up to automorphism of $\pp$ the foliations $\mathcal{F}_{1}$ and $\mathcal{F}_{2}$ are the only foliations that realize the minimal dimension of orbits in degree $3.$}
\end{coralph}

\noindent A.~\textsc{Beltr\'{a}n}, M.~\textsc{Falla~Luza} and D.~\textsc{Mar\'{\i}n} have shown in \cite{BFM13} that $\mathbf{FP}(3)$ contains the set of foliations  $\F\in\mathbf{F}(3)$ whose leaves which are not straight lines do not have inflection points. These foliations are called \textsl{convex}. From these works (\cite[Corollary~4.7]{BFM13}) and from Theorem~\ref{thmalph:classification} we deduce the classification, up to automorphism of $\pp$, of convex foliations of degree $3$ on $\pp.$

\begin{coralph}\label{coralph:class-convexe-3}
{\sl Up to automorphism of $\pp$ there are four convex foliations of degree three on the complex projective plane, namely the foliations  $\mathcal{H}_{1},\mathcal{H}_{\hspace{0.2mm}3},\F_1$ and $\F_3.$
}
\end{coralph}

\noindent This corollary is an analog in degree $3$ of a result on the foliations of degree $2$ due to C.~\textsc{Favre} and J.~\textsc{Pereira} (\cite[Proposition~7.4]{FP15}).
\medskip

\noindent According to~\cite[Theorem~3]{MP13}, we know that the closure in $\mathbf{F}(3)$ of the orbit $\mathcal{O}(\F_3)$ of the \textsc{Fermat} foliation $\F_3$ is an irreducible component of $\mathbf{FP}(3)$. To our knowledge, at present, this is the only explicit example of an irreducible component of $\mathbf{FP}(3)$ appearing in the literature. By analyzing the incidence relations between the closures of the orbits of $\mathcal{H}_{\hspace{0.2mm}i}$ and $\F_j,$ we obtain the decomposition of  $\mathbf{FP}(3)$ into its irreducible components.

\begin{thmalph}\label{thmalph:12-Composantes irréductibles}
{\sl The closures being taken in $\mathbf{F}(3)$ we have
\begin{align*}
&
\overline{\mathcal{O}(\F_1)}=\mathcal{O}(\F_1),&&
\overline{\mathcal{O}(\F_2)}=\mathcal{O}(\F_2),\\
&
\overline{\mathcal{O}(\F_3)}=\mathcal{O}(\F_1)\cup\mathcal{O}(\mathcal{H}_{1})\cup\mathcal{O}(\mathcal{H}_{\hspace{0.2mm}3})\cup\mathcal{O}(\F_3),&&
\overline{\mathcal{O}(\F_4)}=\mathcal{O}(\F_1)\cup\mathcal{O}(\F_2)\cup\mathcal{O}(\F_4),\\
&
\overline{\mathcal{O}(\mathcal{H}_{1})}=\mathcal{O}(\F_1)\cup\mathcal{O}(\mathcal{H}_{1}),&&
\overline{\mathcal{O}(\mathcal{H}_{\hspace{0.2mm}2})}=\mathcal{O}(\F_2)\cup\mathcal{O}(\mathcal{H}_{\hspace{0.2mm}2}),\\
&
\overline{\mathcal{O}(\mathcal{H}_{\hspace{0.2mm}3})}=\mathcal{O}(\F_1)\cup\mathcal{O}(\mathcal{H}_{\hspace{0.2mm}3}),&&
\overline{\mathcal{O}(\mathcal{H}_{\hspace{0.2mm}8})}=\mathcal{O}(\F_2)\cup\mathcal{O}(\mathcal{H}_{\hspace{0.2mm}8}),\\
&
\overline{\mathcal{O}(\mathcal{H}_{\hspace{0.2mm}5})}=\mathcal{O}(\F_1)\cup\mathcal{O}(\mathcal{H}_{\hspace{0.2mm}5}),&&
\overline{\mathcal{O}(\mathcal{H}_{\hspace{0.2mm}4})}\supset\mathcal{O}(\F_2)\cup\mathcal{O}(\mathcal{H}_{\hspace{0.2mm}4}),\\
&
\overline{\mathcal{O}(\mathcal{H}_{\hspace{0.2mm}7})}\supset\mathcal{O}(\F_1)\cup\mathcal{O}(\mathcal{H}_{\hspace{0.2mm}7}),&&
\overline{\mathcal{O}(\mathcal{H}_{\hspace{0.2mm}6})}\supset\mathcal{O}(\F_2)\cup\mathcal{O}(\mathcal{H}_{\hspace{0.2mm}6}),\\
&
\overline{\mathcal{O}(\mathcal{H}_{\hspace{0.2mm}9})}\subset\mathcal{O}(\F_1)\cup\mathcal{O}(\mathcal{H}_{\hspace{0.2mm}9}),&&
\overline{\mathcal{O}(\mathcal{H}_{\hspace{0.2mm}11})}\subset\mathcal{O}(\F_2)\cup\mathcal{O}(\mathcal{H}_{11}),\\
&
\overline{\mathcal{O}(\mathcal{H}_{\hspace{0.2mm}10})}\subset\mathcal{O}(\F_1)\cup\mathcal{O}(\F_2)\cup\mathcal{O}(\mathcal{H}_{\hspace{0.2mm}10}),&&
\overline{\mathcal{O}(\F_5)}\subset\mathcal{O}(\F_2)\cup\mathcal{O}(\F_5)
\end{align*}
with
\begin{align*}
&\dim\mathcal{O}(\F_{1})=6,&&&\dim\mathcal{O}(\F_{2})=6,&&& \dim\mathcal{O}(\mathcal{H}_{\hspace{0.2mm}i})=7, i=1,\ldots,11,\\
&
\dim\mathcal{O}(\F_{4})=7,&&& \dim\mathcal{O}(\F_{5})=7,&&& \dim\mathcal{O}(\F_{3})=8.
\end{align*}
In particular,
\begin{itemize}
  \item the set $\mathbf{FP}(3)$ has exactly twelve irreducible components, namely $\overline{\mathcal{O}(\F_3)},\,$ $\overline{\mathcal{O}(\F_4)},\,$ $\overline{\mathcal{O}(\F_5)},\,$ $\overline{\mathcal{O}(\mathcal{H}_{\hspace{0.2mm}2})},\,$ $\overline{\mathcal{O}(\mathcal{H}_{\hspace{0.4mm}k})},\,k=4,5,\ldots,11$;
  \item the set of convex foliations of degree three in $\pp$ is exactly the closure $\overline{\mathcal{O}(\F_3)}$ of $\mathcal{O}(\F_3)$ (it is therefore an irreducible closed subset of~$\mathbf{F}(3)).$
\end{itemize}
}
\end{thmalph}

\subsection*{Acknowledgment} The authors would like to thank Dominique \textsc{Cerveau} and Frank \textsc{Loray} for their interest in this work.

\bigskip

\section{Preliminaries}
\bigskip

\subsection{Webs}\label{subsec-tissus}

\noindent Let $k\geq1$ be a integer. A {\sl (global) $k$-web} $\mathcal{W}$ on a complex surface $S$ is given by an open covering $(U_{i})_{i\in I}$ of $S$ and a collection of symmetric $k$-forms  $\omega_{i}\in \mathrm{Sym}^{k}\Omega^{1}_{S}(U_{i})$ with isolated zeroes satisfying:
\begin{itemize}
\item [($\mathfrak{a}$)] there exists $g_{ij}\in \mathcal{O}^{*}_{S}(U_{i}\cap U_{j})$ such that $\omega_i$ coincides with $g_{ij}\omega_j$ on $U_i\cap U_j$;
\item [($\mathfrak{b}$)] at every generic point $m$ of $U_i$, $\omega_i(m)$ factorizes as the product of $k$ pairwise non collinear $1$-forms.
\end{itemize}
The subset of points of $S$ not satisfying condition ($\mathfrak{b}$) is called the {\sl discriminant} of $\mathcal{W}$ and it is denoted by $\Delta(\mathcal{W})$. When $k=1$ this condition is always satisfied and we recover the usual definition of a holomorphic foliation on $S$. The cocycle $(g_{ij})$ defines a line bundle $N$ on $S$, which is called {\sl normal bundle} of $\W$, and the local $k$-forms $\omega_i$ patch together to form a global section $\omega\in\mathrm{H}^0(S,\mathrm{Sym}^{k}\Omega^{1}_{S}\otimes N)$.

\noindent A global $k$-web $\W$ on $S$ is said {\sl decomposable} if there are global webs $\mathcal{W}_{1},\mathcal{W}_{2}$ on $S$ with no common subweb such that $\mathcal{W}$ is the superposition of $\mathcal{W}_{1}$ and $\mathcal{W}_{2}$. In this case we will write  $\mathcal{W}=\mathcal{W}_{1}\boxtimes\mathcal{W}_{2}.$ Otherwise it is said that $\W$ is {\sl irreducible}. We will say that $\W$ is {\sl completely decomposable} if there exist global foliations $\F_1,\ldots,\F_k$ on $S$ such that $\W=\F_1\boxtimes\cdots\boxtimes\F_k$. For more details \emph{see} \cite{PP09}.
\smallskip

\noindent In this work we restrict ourselves to the case $S=\pp$. In this case, every $k$-web $\W$ on $\pp$ can be defined in a given affine chart $(x,y)$ by a polynomial $k$-symmetric form $\omega=\sum_{i+j=k}a_{ij}(x,y)\mathrm{d}x^{i}\mathrm{d}y^{j}$, with isolated zeroes and whose discriminant is not identically zero.  Thus, $\W$ is defined by a polynomial differential equation  $F(x,y,y')=0$  of degree $k$ in $y'$. A $k$-web $\W$ on $\pp$ is said of {\sl degree} $d$ if the number of points where a generic line of $\pp$ is tangent to $\W$ is equal to $d$, it is  equivalent to require that $\W$ has normal bundle $N=\mathcal{O}_{\pp}(d+2k)$.
It is well known, \emph{see} for instance  \cite[Proposition~1.4.2]{PP09}, that the webs of degree $0$ are the algebraic webs (whose leaves are the tangent lines of a given reduced algebraic curve).

\noindent The authors in  \cite{MP13} associate to every $k$-web of degree $d\ge 1$ on $\pp$, a $d$-web of degree $k$ on the dual projective plane $\pd$, called the {\sl  \textsc{Legendre} transform} of $\W$ and denoted by $\Leg\W$. 
The leaves of $\Leg\W$ are of the form   $\check{L}=\{\mathrm{T}_{\hspace{-0.4mm}m}L\hspace{1mm}\colon m\in L\}\subset\pd$ where $L\subset\mathbb{P}^2_{\mathbb C}$ is a leaf of $\W$.
More explicitly,  let $(x,y)$ be an affine chart of $\pp$ and consider the affine chart $(p,q)$ of $\pd$ associated to the line $\{y=px-q\}\subset{\mathbb{P}^{2}_{\mathbb{C}}}.$ Let $F(x,y;p)=0$, $p=\frac{\mathrm{d}y}{\mathrm{d}x}$, be an implicit differential equation defining $\W$. Then $\Leg\W$ is given in the affine chart $(p,q)$ of $\pd$ by the implicit differential equation
\[
\check{F}(p,q;x):=F(x,px-q;p)=0, \qquad \text{with} \qquad x=\frac{\mathrm{d}q}{\mathrm{d}p}.
\]
\noindent In particular, if $\F$ is a foliation of degree $d\geq 1$ on $\pp$ defined by a $1$-form  $\omega=A(x,y)\mathrm{d}x+B(x,y)\mathrm{d}y,$ where $A,B\in\mathbb{C}[x,y],$ $\mathrm{pgcd}(A,B)=1$, then $\Leg\F$
is the irreducible $d$-web of degree $1$ on  $\pd$ defined by
\[
A(x,px-q)+pB(x,px-q)=0, \qquad \text{with} \qquad x=\frac{\mathrm{d}q}{\mathrm{d}p}.
\]
\noindent Conversely, every irreducible $d$-web of degree $1$ on $\pd$ is necessarily the \textsc{Legendre} transform of a certain foliation of degree $d$ on $\pp$ (\emph{see}  \cite{MP13}).

\subsection{Curvature and flatness}\label{subsec:courbure-platitude}

We recall here the definition of the curvature of a $k$- web $\W$.
We assume first that $\W$ is a germ of completely decomposable $k$-web on $(\C^2,0)$, $\mathcal{W}=\mathcal{F}_{1}\boxtimes\cdots\boxtimes\mathcal{F}_{k}.$
For each $1\leq i\leq k,$ let $\omega_i$ be a $1$-form defining the foliation $\F_i$ with isolated singularity at $0$. After \cite{PP08},
for each triple  $(r,s,t)$ with $1\leq r<s<t\leq k,$ we define $\eta_{rst}=\eta(\mathcal{F}_{r}\boxtimes\mathcal{F}_{s}\boxtimes\mathcal{F}_{t})$
as the unique meromorphic $1$-form satisfying the following equalities:
\begin{equation*}\label{equa:eta-rst}
{\left\{\begin{array}[c]{lll}
\mathrm{d}(\delta_{st}\,\omega_{r}) &=& \eta_{rst}\wedge\delta_{st}\,\omega_{r}\\
\mathrm{d}(\delta_{tr}\,\omega_{s}) &=& \eta_{rst}\wedge\delta_{tr}\,\omega_{s}\\
\mathrm{d}(\delta_{rs}\,\omega_{t}) &=& \eta_{rst}\wedge\delta_{rs}\,\omega_{t}
\end{array}
\right.}
\end{equation*}
where  $\delta_{ij}$ denotes the function defined by  $\omega_{i}\wedge\omega_{j}=\delta_{ij}\,\mathrm{d}x\wedge\mathrm{d}y.$
Since each $1$-form $\omega_i$ is well defined up to multiplication by an invertible element of  $\mathcal{O}(\mathbb{C}^{2},0),$
it follows that each $1$-form  $\eta_{rst}$ is well defined up to addition of a closed holomorphic $1$-form. Thus, the $1$-form
\begin{equation*}\label{equa:eta}
\hspace{7mm}\eta(\mathcal{W})=\eta(\mathcal{F}_{1}\boxtimes\cdots\boxtimes\mathcal{F}_{k})=\sum_{1\le r<s<t\le k}\eta_{rst}
\end{equation*}
is well defined up to addition of a closed holomorphic $1$-form.
The {\sl curvature} of the web $\mathcal{W}=\mathcal{F}_{1}\boxtimes\cdots\boxtimes\mathcal{F}_{k}$ is by definition the $2$-form
\begin{align*}
&K(\mathcal{W})=K(\mathcal{F}_{1}\boxtimes\cdots\boxtimes\mathcal{F}_{k})=\mathrm{d}\,\eta(\mathcal{W}).
\end{align*}
\noindent It can be checked that  $K(\mathcal{W})$ is a meromorphic $2$-form with poles along the discriminant $\Delta(\mathcal{W})$ of $\mathcal{W},$ canonically associated to $\mathcal{W}$. More precisely, for every dominant holomorphic map $\varphi,$ we have $K(\varphi^{*}\mathcal{W})=\varphi^{*}K(\mathcal{W}).$
\smallskip

\noindent Now, if $\W$ is a (not necessarily completely decomposable) $k$-web on a complex surface $S$ then its pull-back by a suitable Galoisian branched covering is totally decomposable. The invariance of the curvature of this new web by the action of the \textsc{Galois} group of the covering allows to bring it down in a global meromorphic $2$-form on $S$, with poles along the discriminant of $\W$ (\emph{see} \cite{MP13}).
\smallskip

\noindent A $k$-web $\mathcal{W}$ is called {\sl flat} if its curvature  $K(\mathcal{W})$ vanishes identically.
\smallskip

\noindent We recall a formula due to A.~\textsc{Hénaut} \cite{Hen00} which gives the curvature of a planar $3$-web $\W$ given by an implicit differential equation
\begin{align*}
& F(x,y,p):=a_{0}(x,y)p^{3}+a_{1}(x,y)p^{2}+a_{2}(x,y)p+a_{3}(x,y)=0, \qquad   p=\frac{\mathrm{d}y}{\mathrm{d}x}.
\end{align*}
Putting
\begin{small}
\begin{align*}
&\hspace{3cm}
R:=\text{Result}(F,\partial_{p}(F))=
\left| \begin{array}{ccccc}
a_0  & a_1  &  a_2  & a_3  & 0 \\
0    & a_0  &  a_1  & a_2  & a_3   \\
3a_0 & 2a_1 &  a_2  & 0    & 0 \\
0    & 3a_0 & 2a_1  & a_2  & 0  \\
0    & 0    & 3a_0  & 2a_1 & a_2
\end{array} \right|\not\equiv0,
\\
\\
&
\alpha_1=
\left| \begin{array}{ccccc}
\partial_{y}(a_{0})                     & a_0 & -a_0  &  0    &  0   \\
\partial_{x}(a_{0})+\partial_{y}(a_{1}) & a_1 &  0    & -2a_0 &  0    \\
\partial_{x}(a_{1})+\partial_{y}(a_{2}) & a_2 &  a_2  & -a_1  & -3a_0  \\
\partial_{x}(a_{2})+\partial_{y}(a_{3}) & a_3 &  2a_3 &  0    & -2a_1   \\
\partial_{x}(a_{3})                     & 0   &  0    &  a_3  & -a_2
\end{array} \right|
\quad\larger{\text{and}}\quad
\alpha_2=
\left| \begin{array}{ccccc}
0   & \partial_{y}(a_{0})                     & -a_0  &  0    &  0  \\
a_0 & \partial_{x}(a_{0})+\partial_{y}(a_{1}) &  0    & -2a_0 &  0   \\
a_1 & \partial_{x}(a_{1})+\partial_{y}(a_{2}) &  a_2  & -a_1  & -3a_0  \\
a_2 & \partial_{x}(a_{2})+\partial_{y}(a_{3}) &  2a_3 &  0    & -2a_1   \\
a_3 & \partial_{x}(a_{3})                     &  0    &  a_3  & -a_2
\end{array} \right|\hspace{0.5mm},
\end{align*}
\end{small}
\hspace{-1mm}
we have that the curvature of the $3$-web  $\W$ is given by (\cite{Hen00})
\begin{equation}\label{equa:Formule-Henaut}
K(\W)=\left(\partial_{y}\left(\frac{\alpha_1}{R}\right)-\partial_{x}\left(\frac{\alpha_2}{R}\right)\right)\mathrm{d}x\wedge\mathrm{d}y.
\end{equation}

\subsection{Singularities and inflection divisor of a foliation on the projective plane}

A degree $d$ holomorphic foliation $\F$ on $\pp$ is defined in homogeneous coordinates $(x,y,z)$ by a $1$-form
$$\omega=a(x,y,z)\mathrm{d}x+b(x,y,z)\mathrm{d}y+c(x,y,z)\mathrm{d}z,$$
where $a,$ $b$ and $c$ are homogeneous polynomials of degree $d+1$
without common factor and satisfying the \textsc{Euler} condition $i_{\mathrm{R}}\omega=0$, where $\mathrm{R}=x\frac{\partial{}}{\partial{x}}+y\frac{\partial{}}{\partial{y}}+z\frac{\partial{}}{\partial{z}}$
denotes the radial vector field and  $i_{\mathrm{R}}$ is the interior product by  $\mathrm{R}$. The {\sl singular locus} $\mathrm{Sing}\mathcal{F}$ of $\mathcal{F}$ is the projectivization of the singular locus of~$\omega$ $$\mathrm{Sing}\,\omega=\{(x,y,z)\in\mathbb{C}^3\,\vert \, a(x,y,z)=b(x,y,z)=c(x,y,z)=0\}.$$

\noindent Let $\mathcal{C}\subset\pp$ be an algebraic curve with homogeneous equation $F(x,y,z)=0.$ We say that $\mathcal{C}$ is an \textsl{invariant curve} by $\F$ if $\mathcal{C}\smallsetminus\Sing\F$ is a union of (ordinary) leaves of the regular foliation $\F|_{\pp\smallsetminus\Sing\F}$. In algebraic terms, this is equivalent to require that the $2$-form  $\omega\wedge\mathrm{d}F$ is divisible by $F$, {\it i.e.} it vanishes along each irreducible component of $\mathcal{C}.$

\noindent When each irreducible component of $\mathcal{C}$ is not $\F$-invariant, for every point $p$ of  $\mathcal{C}$ we define the \textsl{tangency order} $\Tang(\F,\mathcal{C},p)$ of $\F$ with $\mathcal{C}$ at $p$ as follows. We fix a local chart $(\mathrm{u},\mathrm{v})$  such that $p=(0,0)$; let $f(\mathrm{u},\mathrm{v})=0$ be a reduced local equation of  $\mathcal{C}$ at a neighborhood of $p$ and let $\X$ be a vector field defining the germ of $\F$ at $p$. We denote by $\X(f)$ the \textsc{Lie} derivative of $f$ with respect to $\X$ and by  $\langle f,\X(f)\rangle$ the ideal of  $\C\{\mathrm{u},\mathrm{v}\}$ generated by $f$ and $\X(f)$. Then
$$\Tang(\F,\mathcal{C},p)=\dim_\mathbb{C}\frac{\C\{\mathrm{u},\mathrm{v}\}}{\langle f,\X(f)\rangle}.$$
It is easy to see that this definition is well-posed, and that
 $\Tang(\F,\mathcal{C},p)<+\infty$ because~$\mathcal{C}$ is not $\F$-invariant.
\smallskip

\noindent Let us recall some local notions attached to the pair
$(\mathcal{F},s)$, where  $s\in\Sing\mathcal{F}$. The germ of $\F$ at $s$ is defined, up to multiplication by a unity in the local ring  $\mathcal{O}_s$ at $s$, by a vector field
\begin{small}
$\mathrm{X}=A(\mathrm{u},\mathrm{v})\frac{\partial{}}{\partial{\mathrm{u}}}+B(\mathrm{u},\mathrm{v})\frac{\partial{}}{\partial{\mathrm{v}}}$.
\end{small}
\noindent The {\sl algebraic multiplicity} $\nu(\mathcal{F},s)$ of $\mathcal{F}$ at $s$ is given by  $$\nu(\mathcal{F},s)=\min\{\nu(A,s),\nu(B,s)\},$$ where $\nu(g,s)$ denotes the algebraic multiplicity of the algebraic function  $g$ at $s$. Let us denote by $\mathfrak{L}_s$ the family of straight lines through $s$ which are not invariant by $\F$. For every line $\ell_s$ of~$\mathfrak{L}_s,$ we have the inequalities $1\leq\Tang(\F,\ell_s,s)\leq d$. This allows us to associate to the pair~$(\mathcal{F},s)$ the following natural (invariant) integers
\begin{align*}
&\hspace{0.8cm}\tau(\mathcal{F},s)=\min\{\Tang(\F,\ell_s,s)\hspace{1mm}\vert\hspace{1mm}\ell_s\in\mathfrak{L}_s\},&&\hspace{0.5cm}
\kappa(\mathcal{F},s)=\max\{\Tang(\F,\ell_s,s)\hspace{1mm}\vert\hspace{1mm}\ell_s\in\mathfrak{L}_s\}.
\end{align*}
The invariant $\tau(\mathcal{F},s)$ represents the tangency order of  $\mathcal{F}$ with a generic line passing through $s$. It is easy to see that
$$\tau(\mathcal{F},s)=\min\{k\geq 1\hspace{1mm}\colon\det(J^{k}_{s}\,\mathrm{X},\mathrm{R}_{s})\neq0\}\geq\nu(\mathcal{F},s),$$
 where $J^{k}_{s}\,\mathrm{X}$  denotes the  $k$-jet of $\mathrm{X}$ at $s$ and $\mathrm{R}_{s}$ is the radial vector field centered at $s$.
\smallskip

\noindent The singularity $s$ is called  {\sl radial of order} $n-1$ if $\nu(\mathcal{F},s)=1$ and $\tau(\mathcal{F},s)=n.$
\smallskip

\noindent The  \textsl{\textsc{Milnor} number} of $\F$ at $s$ is the integer $$\mu(\mathcal{F},s)=\dim_\mathbb{C}\mathcal{O}_s/\langle A,B\rangle,$$ where $\langle A,B\rangle$ denotes the ideal of $\mathcal{O}_s$ generated by $A$ and $B$.
\smallskip

\noindent The singularity $s$ is called {\sl non-degenerate} if  $\mu(\F,s)=1$, or equivalently if the linear part  $J^{1}_{s}\mathrm{X}$ of $\mathrm{X}$ possesses two non-zero eigenvalues  $\lambda,\mu$.
In this case, the quantity  $\mathrm{BB}(\F,s)=\frac{\lambda}{\mu}+\frac{\mu}{\lambda}+2$ is called the {\sl  \textsc{Baum-Bott} invariant} of $\F$ at $s$ (\emph{see} \cite{BB72}). 
By Briot-Bouquet's Theorem (\emph{see} \cite{CS82} for a generalization to any singularity) there is at least a germ of curve
$\mathcal{C}$ at $s$ which is invariant by $\F$. Up to local diffeomorphism we can assume that $s=(0,0)$\, $\mathrm{T}_{s}\mathcal{C}=\{\mathrm{v}=\,0\,\}$ and $J^{1}_{s}\mathrm{X}=\lambda \mathrm{u}\frac{\partial}{\partial\mathrm{u}}+(\varepsilon \mathrm{u}+\mu\hspace{0.1mm}\mathrm{v})\frac{\partial}{\partial \mathrm{v}}$, where we can take $\varepsilon=0$ if $\lambda\neq\mu$. The quantity $\mathrm{CS}(\F,\mathcal{C},s)=\frac{\lambda}{\mu}$ is called the {\sl \textsc{Camacho-Sad} index} of $\F$ at $s$ along $\mathcal{C}$.
\medskip

\noindent Finally, let us recall the notion of inflection divisor of $\F$. Let $\mathrm{Z}=E\frac{\partial}{\partial x}+F\frac{\partial}{\partial y}+G\frac{\partial}{\partial z}$ be a homogeneous vector field of degree $d$ on  $\mathbb{C}^3$ non collinear to the radial vector field describing $\mathcal{F},$ {\it i.e.} such that $\omega=i_{\mathrm{R}}i_{\mathrm{Z}}\mathrm{d}x\wedge\mathrm{d}y\wedge\mathrm{d}z.$ The {\sl inflection divisor} of $\mathcal{F}$, denoted by $\IF$, is the divisor of $\pp$ defined by the homogeneous equation
\begin{equation}\label{equa:ext1}
\left| \begin{array}{ccc}
x &  E &  \mathrm{Z}(E) \\
y &  F &  \mathrm{Z}(F)  \\
z &  G &  \mathrm{Z}(G)
\end{array} \right|=0.
\end{equation}
This divisor has been studied in  \cite{Per01} in a more general context.
In particular, the following properties has been proved.
\begin{enumerate}
\item [\texttt{1.}] On $\mathbb{P}^{2}_{\mathbb{C}}\smallsetminus\mathrm{Sing}\mathcal{F},$ $\IF$ coincides with the curve described by the inflection
points of the leaves of $\mathcal{F}$;
\item [\texttt{2.}] If $\mathcal{C}$ is an irreducible algebraic curve invariant by $\mathcal{F}$ then $\mathcal{C}\subset\IF$ if and only if $\mathcal{C}$ is an invariant line;
\item [\texttt{3.}] $\IF$ can be decomposed into $\IF=\IinvF+\ItrF,$ where
the support of  $\IinvF$ consists in the set of invariant lines of $\F$ and the support of  $\ItrF$ is the closure of the inflection points along the leaves of  $\mathcal{F}$  which are not lines;
\item [\texttt{4.}] The degree of the divisor $\IF$ is $3d.$
\end{enumerate}

\noindent The foliation $\mathcal{F}$ will be called {\sl convex} if its inflection divisor  $\IF$ is totally invariant by  $\mathcal{F}$, {\it i.e.} if $\IF$
is a product of invariant lines.

\section{Foliations of $\mathbf{FP}(3)$ having at least one degenerate singularity}\label{sec:cas-dégénéré}
\medskip

\Subsection{Case of a degenerate singularity of algebraic multiplicity at most $2$}

\noindent In \cite[Appendix~A]{Bed17} the first author gives a computational proof of the following statement.

\begin{pro}\label{pro:cas-nilpotent-noeud-ordre-2}
{\sl Let $\mathcal{F}$ be a degree three foliation on $\pp$ having a degenerate singularity of algebraic multiplicity at most $2$. Then the dual $3$-web $\Leg\F$ of $\F$ is not flat.}
\end{pro}

\noindent In this appendix the matter is about a proof by contradiction: first, the author assumes that there is a foliation $\F$ of degree $3$ on $\pp$ such that the $3$-web $\Leg\F$ is flat and such that the singular locus $\Sing\F$ contains a point $m$ satisfying $\mu(\F,m)\geq2$\, and \,$\nu(\F,m)\leq2$; then, he explicitly calculates the curvature of $\Leg\F$ by Formula~(\ref{equa:Formule-Henaut}) and he shows that the condition $K(\Leg\F)\equiv0$ contradicts the hypothesis $\deg\F=3.$

\begin{prob}\label{prob:cas-nilpotent-noeud-ordre-2}
{\sl
Give a non-computational proof of Proposition~\ref{pro:cas-nilpotent-noeud-ordre-2}.
}
\end{prob}

\Subsection{Case of a degenerate singularity of algebraic multiplicity $3$}

\noindent In this paragraph we are interested in the foliations $\F\in\mathbf{FP}(3)$ which have a degenerate singularity $m$ of algebraic multiplicity $3.$ We distinguish two cases according to whether $\Sing\F=\{m\}$ or $\{m\}\varsubsetneq\Sing\F.$

\subsubsection{The singular locus is reduced to a point of algebraic multiplicity $3$}
We start by establishing the following statement classifying the foliations of $\mathbf{F}(3)$ whose singular locus is reduced to a point of algebraic multiplicity $3.$
\begin{pro}\label{pro:v=3-mu=13}
{\sl Let $\F$ be a foliation of degree $3$ on $\pp$ with exactly one singularity. Let $\omega$ be a $1$-form defining $\F.$ If this singularity is of algebraic multiplicity $3$, then up to isomorphism $\omega$ is of one of the following types
\smallskip
\begin{itemize}
\item[\texttt{1. }]\hspace{1mm} $x^3\mathrm{d}x+y^2(c\hspace{0.1mm}x+y)(x\mathrm{d}y-y\mathrm{d}x),\hspace{1mm}c\in\mathbb{C}$;

\item[\texttt{2. }]\hspace{1mm} $x^3\mathrm{d}x+y(x+c\hspace{0.1mm}xy+y^2)(x\mathrm{d}y-y\mathrm{d}x),\hspace{1mm}c\in\mathbb{C}$;

\item[\texttt{3. }]\hspace{1mm} $x^3\mathrm{d}x+(x^2+c\hspace{0.1mm}xy^2+y^3)(x\mathrm{d}y-y\mathrm{d}x),\hspace{1mm}c\in\mathbb{C}$;

\item[\texttt{4. }]\hspace{1mm} $x^2y\mathrm{d}x+(x^3+c\hspace{0.1mm}xy^2+y^3)(x\mathrm{d}y-y\mathrm{d}x),\hspace{1mm}c\in\mathbb{C}$;

\item[\texttt{5. }]\hspace{1mm} $x^2y\mathrm{d}x+(x^3+\delta\,xy+y^3)(x\mathrm{d}y-y\mathrm{d}x),\hspace{1mm}\delta\in\mathbb{C}^*$;

\item[\texttt{6. }]\hspace{1mm} $x^2y\mathrm{d}y+(x^3+c\hspace{0.1mm}xy^2+y^3)(x\mathrm{d}y-y\mathrm{d}x),\hspace{1mm} c\in\mathbb{C}$;

\item[\texttt{7. }]\hspace{1mm} $xy(x\mathrm{d}y-\lambda\,y\mathrm{d}x)
                                 +(x^3+y^3)(x\mathrm{d}y-y\mathrm{d}x),\hspace{1mm}\lambda\in\mathbb{C}\setminus\{0,1\}$;

\item[\texttt{8. }]\hspace{1mm} $xy(y-x)\mathrm{d}x+(c_0\,x^3+c_1x^2y+y^3)(x\mathrm{d}y-y\mathrm{d}x),\hspace{1mm} c_0(c_0+c_1+1)\neq0$.
\end{itemize}
\smallskip

\noindent These eight $1$-forms are not linearly conjugated with each other.
}
\end{pro}

\noindent This proposition is an analog in degree $3$ of a result on the foliations of degree $2$ due to D.~\textsc{Cerveau}, J.~\textsc{D\'eserti}, D.~\textsc{Garba Belko} and R.~\textsc{Meziani} (\cite[Proposition 1.8]{CDGBM10}). The proof that we will give is very close to that of~\cite{CDGBM10}; it will result from Lemmas~\ref{lem:exclure-cône-4-droites},~\ref{lem:cône-3-droites},~\ref{lem:cône-2-droites} and \ref{lem:cône-1-droite} stated below. In these four lemmas $\mathcal{F}$ denotes a foliation of degree three on $\pp$ defined by a $1$-form $\omega$ and such that
\begin{itemize}
\item[\texttt{1. }] the unique singularity of $\F$ is $O=[0:0:1]$;
\item[\texttt{2. }] the jets of order $1$ and $2$ of $\omega$ at $(0,0)$ are zero, {\it i.e.} $\nu(\F,O)=3$.
\end{itemize}
In this case
$$\omega=A(x,y)\mathrm{d}x+B(x,y)\mathrm{d}y+C(x,y)(x\mathrm{d}y-y\mathrm{d}x),$$
where $A,$ $B$ and $C$ are homogeneous polynomials of degree $3.$ The foliation $\F$ being of degree three, the tangent cone $xA+yB$ of $\omega$ at $(0,0)$ can not be identically zero. The polynomial $C$ is also not identically zero, otherwise the line at infinity would be invariant by $\mathcal{F}$ which would therefore have a singularity on this line, which is excluded. We will reason according to the nature of the tangent cone which, a priori, can be four lines, three lines, two lines or a single line.
\begin{lem}\label{lem:exclure-cône-4-droites}
{\sl
Every irreducible factor $L$ of $xA+yB$ divides $\gcd(A,B)$ and does not divide~$C.$ In particular, the tangent cone of $\omega$ at $(0,0)$ is not the union of four distinct lines.
}
\end{lem}

\begin{proof}
Up to isomorphism, we can assume that $L=x$; then $x$ divides $B.$ Thus on the line $x=0$ the form $\omega$ writes as $A(0,y)\mathrm{d}x-y\,C(0,y)\mathrm{d}x=y^3\left(A(0,1)-y\,C(0,1)\right)\mathrm{d}x.$ Since $O$ is the unique singularity of $\mathcal{F},$ the product $A(0,1)C(0,1)$ is zero. The point $[0:1:0]$ being non-singular, $C(0,1)$ is non-zero and as a result $A(0,1)=0$; hence $x$ divides $A$ but not $C.$
\end{proof}

\begin{lem}\label{lem:cône-3-droites}
{\sl If the tangent cone of $\omega$ at $(0,0)$ is the union of three distinct lines, then up to automorphisms of $\mathbb P^2_{\mathbb C}$,~$\omega$ is of type
\begin{align*}
&\hspace{2cm} xy(y-x)\mathrm{d}x+(c_0\,x^3+c_1x^2y+y^3)(x\mathrm{d}y-y\mathrm{d}x),&& c_0(c_0+c_1+1)\neq0.
\end{align*}
}
\end{lem}

\begin{proof}
We can assume that $xA+yB=\ast x^2y(y-x),\ast\in\mathbb{C}^{*}$; it follows that $\omega$ writes (\cite{CM82})
\begin{small}
\begin{align*}
& x^2y(y-x)\left(\lambda_0\frac{\mathrm{d}x}{x}+\lambda_1\frac{\mathrm{d}y}{y}+\lambda_2\frac{\mathrm{d}(y-x)}{y-x}
+\delta\hspace{0.1mm}\mathrm{d}\left(\frac{y}{x}\right)\right)+(c_0\,x^3+c_1x^2y+c_2xy^2+c_3y^3)(x\mathrm{d}y-y\mathrm{d}x),&& \delta,\lambda_i,c_i\in\mathbb{C}.
\end{align*}
\end{small}
\hspace{-1mm}Therefore we have
\begin{align*}
&A(x,y)=y\Big((y-x)(\lambda_0\,x-\delta y)-\lambda_2x^2\Big),&&
B(x,y)=x\Big((y-x)(\lambda_1x+\delta y)+\lambda_2xy\Big),&&
C(x,y)=\sum_{i=0}^{3}c_i\hspace{0.1mm}x^{3-i}y^{i}.
\end{align*}
According to Lemma~\ref{lem:exclure-cône-4-droites}, the polynomial $xy(y-x)$ divides $A$ and $B$ but not $C$, which means that
\begin{align*}
& A(0,1)=A(1,0)=A(1,1)=B(0,1)=B(1,0)=B(1,1)=0&& \text{and} && C(0,1)C(1,0)C(1,1)\neq0.
\end{align*}
It follows that $\delta=\lambda_1=\lambda_2=0$\, and that \,$c_0c_3(c_0+c_1+c_2+c_3)\neq0$. The foliation $\F$ being of degree three $\lambda_0$~is non-zero; we can therefore assume that~$\lambda_0=1$, hence
\begin{align*}
& \omega=xy(y-x)\mathrm{d}x+(c_0\,x^3+c_1x^2y+c_2xy^2+c_3y^3)(x\mathrm{d}y-y\mathrm{d}x).
\end{align*}
After conjugating $\omega$ by the homothety $\left(\frac{1}{c_3}x,\frac{1}{c_3}y\right)$, we can normalize the coefficient $c_3$ to $1$; as a consequence
\begin{align*}
& \omega=xy(y-x)\mathrm{d}x+(c_0\,x^3+c_1x^2y+c_2xy^2+y^3)(x\mathrm{d}y-y\mathrm{d}x),&& c_0(c_0+c_1+c_2+1)\neq0.
\end{align*}
The conjugation by the automorphism~\begin{small}$\left(\dfrac{x}{1+c_2y},\dfrac{y}{1+c_2y}\right)$\end{small} of $\mathbb{P}^2_{\mathbb C}$ allows us to cancel $c_2$. Hence the statement holds.
\end{proof}

\begin{lem}\label{lem:cône-2-droites}
{\sl
If the tangent cone of $\omega$ at $(0,0)$ is composed of two distinct lines, then up to isomorphism $\omega$ is of one of the following types
\begin{itemize}
\item[\texttt{1. }]\hspace{1mm} $x^2y\mathrm{d}x+(x^3+c\hspace{0.1mm}xy^2+y^3)(x\mathrm{d}y-y\mathrm{d}x),\hspace{1mm} c\in\mathbb{C}$;

\item[\texttt{2. }]\hspace{1mm} $x^2y\mathrm{d}x+(x^3+\delta\,xy+y^3)(x\mathrm{d}y-y\mathrm{d}x),\hspace{1mm} \delta\in\mathbb{C}^*$;

\item[\texttt{3. }]\hspace{1mm} $x^2y\mathrm{d}y+(x^3+c\hspace{0.1mm}xy^2+y^3)(x\mathrm{d}y-y\mathrm{d}x),\hspace{1mm} c\in\mathbb{C}$;

\item[\texttt{4. }]\hspace{1mm} $xy(x\mathrm{d}y-\lambda\,y\mathrm{d}x)+(x^3+y^3)(x\mathrm{d}y-y\mathrm{d}x),\hspace{1mm}
                                 \lambda\in\mathbb{C}\setminus\{0,1\}$.
\end{itemize}
}
\end{lem}

\begin{proof}
Up to linear conjugation we are in one of the two following situations
\begin{itemize}
\item[($\mathfrak{a}$)] $xA+yB=\ast x^3y,\quad \ast\in\mathbb{C}^{*}$;

\item[($\mathfrak{b}$)] $xA+yB=\ast x^2y^2,\quad\hspace{-1.7mm} \ast\in\mathbb{C}^{*}$.
\end{itemize}
Let us start by studying the eventuality ($\mathfrak{a}$). In this case the $1$-form $\omega$ writes (\cite{CM82})
\begin{align*}
& x^3y\left(\lambda_0\frac{\mathrm{d}x}{x}+\lambda_1\frac{\mathrm{d}y}{y}+\mathrm{d}\left(\frac{\delta_1xy+\delta_2y^2}{x^2}\right)\right)
+(c_0\,x^3+c_1x^2y+c_2xy^2+c_3y^3)(x\mathrm{d}y-y\mathrm{d}x),&& \lambda_i,\delta_i,c_i\in\mathbb{C}.
\end{align*}
Then we have
\begin{align*}
&A(x,y)=y(\lambda_0x^2-\delta_1xy-2\delta_2y^2),&&
B(x,y)=x(\lambda_1x^2+\delta_1xy+2\delta_2y^2),&&
C(x,y)=\sum_{i=0}^{3}c_i\hspace{0.1mm}x^{3-i}y^{i}.
\end{align*}
According to Lemma~\ref{lem:exclure-cône-4-droites}, the polynomial $xy$ divides $A$ and $B$ but not $C$. As a result $\delta_2=\lambda_1=0$\, and \,$c_0c_3\neq0.$ The foliation $\F$ being of degree three the coefficient $\lambda_0$ is non-zero and we can assume it equals $1.$ Thus $\F$ is described by \begin{align*}
& \omega=x^2y\mathrm{d}x+\delta_1xy(x\mathrm{d}y-y\mathrm{d}x)+(c_0\,x^3+c_1x^2y+c_2xy^2+c_3y^3)(x\mathrm{d}y-y\mathrm{d}x).
\end{align*}
The diagonal linear transformation \begin{small}$\left(\dfrac{1}{c_0}x,\sqrt[3]{\dfrac{1}{c_3c_0^2}}y\right)$\end{small} allows us to assume that $c_0=c_3=1$; as a consequence
\begin{align*}
& \omega=x^2y\mathrm{d}x+(x^3+\delta_1xy+c_1x^2y+c_2xy^2+y^3)(x\mathrm{d}y-y\mathrm{d}x).
\end{align*}
If $\delta_1=0$, resp. $\delta_1\neq0$, by conjugating $\omega$ by
\begin{align*}
& \left(\frac{x}{1+c_1y},\frac{y}{1+c_1y}\right),&&
\text{resp.}\hspace{1mm}
\left(
\frac{x}{1-\left(\dfrac{c_2}{\delta_1}\right)y-\left(\dfrac{\delta_1c_1+c_2}{\delta_1^2}\right)x},
\frac{y}{1-\left(\dfrac{c_2}{\delta_1}\right)y-\left(\dfrac{\delta_1c_1+c_2}{\delta_1^2}\right)x}
\right),
\end{align*}
we reduce ourselves to $c_1=0$, resp. $c_1=c_2=0$, that is, to
\begin{align*}
& \omega=x^2y\mathrm{d}x+(x^3+c_2xy^2+y^3)(x\mathrm{d}y-y\mathrm{d}x),&&
\text{resp.}\hspace{1mm}
\omega=x^2y\mathrm{d}x+(x^3+\delta_1xy+y^3)(x\mathrm{d}y-y\mathrm{d}x);
\end{align*}
hence the two first announced models.
\vspace{2mm}

\noindent Let us now consider the possibility ($\mathfrak{b}$). In this case $\omega$ writes (\cite{CM82})
\begin{align*}
& x^2y^2\left(\lambda_0\frac{\mathrm{d}x}{x}+\lambda_1\frac{\mathrm{d}y}{y}+\mathrm{d}\left(\frac{\delta_1x^2+\delta_2y^2}{xy}\right)\right)
+(c_0\,x^3+c_1x^2y+c_2xy^2+c_3y^3)(x\mathrm{d}y-y\mathrm{d}x),&& \lambda_i,\delta_i,c_i\in\mathbb{C}.
\end{align*}
Here $A(x,y)=y(\delta_1x^2+\lambda_0\,xy-\delta_2y^2)$\, and \,$B(x,y)=x(\delta_2y^2+\lambda_1xy-\delta_1x^2)$. According to Lemma~\ref{lem:exclure-cône-4-droites} again, $xy$ divides $\mathrm{pgcd}(A,B)$ and does not divide $C$, which is equivalent to $\delta_1=\delta_2=0$\, and \,$c_0c_3\neq0.$

\noindent The foliation $\F$ being of degree three the sum $\lambda_0+\lambda_1$ is non-zero; then one of the coefficients $\lambda_i$ is non-zero and we can obviously normalize it to $1.$ Since the lines of the tangent cone ({\it i.e.} $x=0$ and $y=0$) play a symmetrical role, it suffices to treat the eventuality $\lambda_1=1.$ Thus $\F$ is given by
\begin{align*}
&\hspace{1cm}\omega=xy(x\mathrm{d}y+\lambda_0y\mathrm{d}x)+(c_0\,x^3+c_1x^2y+c_2xy^2+c_3y^3)(x\mathrm{d}y-y\mathrm{d}x),&& (\lambda_0+1)c_0c_3\neq0.
\end{align*}
Let $\alpha$ be in $\mathbb{C}$ such that $\alpha^3=\dfrac{1}{c_3c_0^2}$; let us put $\beta=c_0\alpha^2.$ After conjugating $\omega$ by $\left(\alpha\,x,\beta y\right)$, we can assume that $c_0=c_3=1$; as a result
\begin{align*}
&\hspace{1cm}\omega=xy(x\mathrm{d}y+\lambda_0y\mathrm{d}x)+(x^3+c_1x^2y+c_2xy^2+y^3)(x\mathrm{d}y-y\mathrm{d}x),&& \lambda_0+1\neq0.
\end{align*}
If $\lambda_0=0$, resp. $\lambda_0\neq0$, by conjugating $\omega$ by
\begin{align*}
& \left(\frac{x}{1-c_1x},\frac{y}{1-c_1x}\right),&&
\text{resp.}\hspace{1mm}
\left(\frac{x}{1+\left(\dfrac{c_2}{\lambda_0}\right)y-c_1x},\frac{y}{1+\left(\dfrac{c_2}{\lambda_0}\right)y-c_1x}\right),
\end{align*}
we reduce ourselves to $c_1=0$, resp. $c_1=c_2=0$, that is, to
\begin{align*}
\hspace{1cm}
&\omega=x^2y\mathrm{d}y+(x^3+c_2xy^2+y^3)(x\mathrm{d}y-y\mathrm{d}x),&&\\
\hspace{1cm}
\text{resp.}\hspace{1mm}
&\omega=xy(x\mathrm{d}y+\lambda_0y\mathrm{d}x)+(x^3+y^3)(x\mathrm{d}y-y\mathrm{d}x),\quad \lambda_0(\lambda_0+1)\neq0,
\end{align*}
which are the two last announced models.
\end{proof}

\begin{lem}\label{lem:cône-1-droite}
{\sl
If the tangent cone of $\omega$ at $(0,0)$ is reduced to a single line, then up to isomorphism $\omega$ is of one of the following types
\begin{itemize}
\item[\texttt{1. }]\hspace{1mm} $x^3\mathrm{d}x+y^2(c\hspace{0.1mm}x+y)(x\mathrm{d}y-y\mathrm{d}x),\hspace{1mm}c\in\mathbb{C}$;

\item[\texttt{2. }]\hspace{1mm} $x^3\mathrm{d}x+y(x+c\hspace{0.1mm}xy+y^2)(x\mathrm{d}y-y\mathrm{d}x),\hspace{1mm}c\in\mathbb{C}$;

\item[\texttt{3. }]\hspace{1mm} $x^3\mathrm{d}x+(x^2+c\hspace{0.1mm}xy^2+y^3)(x\mathrm{d}y-y\mathrm{d}x),\hspace{1mm}c\in\mathbb{C}$.
\end{itemize}
}
\end{lem}

\begin{proof}
We can assume that the tangent cone is the line $x=0$; then $\omega$ writes as
\begin{align*}
& x^4\left(\lambda\frac{\mathrm{d}x}{x}
+\mathrm{d}\left(\frac{\delta_1x^2y+\delta_2xy^2+\delta_3y^3}{x^3}\right)\right)
+(c_0\,x^3+c_1x^2y+c_2xy^2+c_3y^3)(x\mathrm{d}y-y\mathrm{d}x),&&
\lambda,\delta_i,c_i\in\mathbb{C}.
\end{align*}
Then we have
\begin{align*}
&A(x,y)=\lambda\,x^3-\delta_1x^2y-2\delta_2xy^2-3\delta_3y^3,&&
B(x,y)=x(\delta_1x^2+2\delta_2xy+3\delta_3y^2),&&
C(x,y)=\sum_{i=0}^{3}c_i\hspace{0.1mm}x^{3-i}y^{i}.
\end{align*}
According to Lemma~\ref{lem:exclure-cône-4-droites}, $x$ divides $A$ and $B$ but not $C$; as a result~$\delta_3=0$ and $c_3\neq0$. The foliation $\F$ being of degree three the coefficient $\lambda$ is non-zero and we can assume that $\lambda=1.$ The conjugation by the homothety $\left(\frac{1}{c_3}x,\frac{1}{c_3}y\right)$ allows us to assume that~$c_3=1.$ Thus $\F$ is described by
\begin{align*}
&& \omega=x^3\mathrm{d}x+(\delta_1x^2+2\delta_2xy+c_0\,x^3+c_1x^2y+c_2xy^2+y^3)(x\mathrm{d}y-y\mathrm{d}x).
\end{align*}
We have the following three possibilities to study
\begin{itemize}
\item[--] $\delta_2\neq0$;

\item[--] $\delta_1=\delta_2=0$;

\item[--] $\delta_2=0,\delta_1\neq0$.
\end{itemize}
\vspace{2mm}

\textbf{\textit{1.}} If $\delta_2\neq0$, then by conjugating $\omega$ by $\left(\alpha^2x,\alpha^{3/2}y-\alpha\delta_1x\right)$, where $\alpha=2\delta_2,$ we reduce ourselves to $\delta_1=0$ and $\delta_2=\frac{1}{2}.$ As a result $\F$ is given by
\begin{align*}
&& \omega=x^3\mathrm{d}x+(xy+c_0\,x^3+c_1x^2y+c_2xy^2+y^3)(x\mathrm{d}y-y\mathrm{d}x).
\end{align*}
The conjugation by the diffeomorphism $\left(\dfrac{x}{1+c_0y-c_1x},\dfrac{y}{1+c_0y-c_1x}\right)$ allows us to assume that $c_0=c_1=0$; as a consequence
\begin{align*}
&& \omega=x^3\mathrm{d}x+y(x+c_2xy+y^2)(x\mathrm{d}y-y\mathrm{d}x),
\end{align*}
hence the second announced model.

\textbf{\textit{2.}} If $\delta_1=\delta_2=0$ the $1$-form $\omega$ writes
\begin{align*}
&& \omega=x^3\mathrm{d}x+(c_0\,x^3+c_1x^2y+c_2xy^2+y^3)(x\mathrm{d}y-y\mathrm{d}x).
\end{align*}
\noindent Let $\alpha$ be in $\mathbb{C}$ such that $3\alpha^2+2c_2\alpha+c_1=0.$ After conjugating $\omega$ by $\left(x,y+\alpha\,x\right)$, we can assume that $c_1=0.$ Then the conjugation by the diffeomorphism $\left(\dfrac{x}{1+c_0y},\dfrac{y}{1+c_0y}\right)$ allows us to cancel $c_0$; hence the first announced model.
\vspace{2mm}

\textbf{\textit{3.}} When $\delta_2=0$ and $\delta_1\neq0$, the form $\omega$ writes
\begin{align*}
&& \omega=x^3\mathrm{d}x+(\delta_1x^2+c_0\,x^3+c_1x^2y+c_2xy^2+y^3)(x\mathrm{d}y-y\mathrm{d}x).
\end{align*}
By conjugating $\omega$ by $\left(\delta_1^4x,\delta_1^3y\right)$, we can assume that $\delta_1=1.$ Then by conjugating by $$\left(\dfrac{x}{1-c_1y-(c_0+c_1)x},\dfrac{y}{1-c_1y-(c_0+c_1)x}\right),$$ we reduce ourselves to $c_0=c_1=0$, that is, to the third announced model.
\end{proof}

\begin{proof}[\sl Proof of Proposition~\ref{pro:v=3-mu=13}] It suffices to choose affine coordinates $(x,y)$ such that the point $(0,0)$ is singular of $\F$ and to use Lemmas \ref{lem:exclure-cône-4-droites}, \ref{lem:cône-3-droites}, \ref{lem:cône-2-droites}, \ref{lem:cône-1-droite}.
\end{proof}

\noindent We are now ready to describe up to isomorphism the foliations of $\mathbf{FP}(3)$ whose singular locus is reduced to a point of algebraic multiplicity~$3.$
\begin{pro}\label{pro:overline-omega2}
{\sl Let $\F$ be a foliation of degree $3$ on $\pp$ with exactly one singularity. Assume that this singularity is of algebraic multiplicity $3$ and that the $3$-web $\Leg\F$ is flat. Then $\F$ is linearly conjugated to the foliation $\F_2$ described by the $1$-form $$\omegaoverline_{2}=x^{3}\mathrm{d}x+y^{3}(x\mathrm{d}y-y\mathrm{d}x).$$
}
\end{pro}

\begin{proof}
Let $\omega$ be a $1$-form describing $\F$ in an affine chart $(x,y)$ and let $(p,q)$ be the affine chart of $\pd$ corresponding to the line $\{px-qy=1\}\subset\pp.$ Up to linear conjugation $\omega$ is of one of the eight types of Proposition~\ref{pro:v=3-mu=13}.
\begin{itemize}
      \smallskip
  \item [--] If $\omega=x^3\mathrm{d}x+y^2(c\hspace{0.1mm}x+y)(x\mathrm{d}y-y\mathrm{d}x),\,c\in\mathbb{C}$, then the $3$-web $\Leg\F$ is given by the differential equation $q(q')^3+cq'+1=0$, where $q'=\frac{\mathrm{d}q}{\mathrm{d}p}$. The explicit computation of $K(\Leg\F)$ leads to $$K(\Leg\F)=-\dfrac{4c^2(2c^3+27q)}{q^2(4c^3+27q)^2}\mathrm{d}p\wedge\mathrm{d}q;$$
      as a result $\Leg\F$ is flat if and only if $c=0$, in which case $\omega=\omegaoverline_{2}=x^{3}\mathrm{d}x+y^{3}(x\mathrm{d}y-y\mathrm{d}x).$
      \smallskip
  \item [--] If $\omega=x^3\mathrm{d}x+y(x+c\hspace{0.1mm}xy+y^2)(x\mathrm{d}y-y\mathrm{d}x),\,c\in\mathbb{C}$, then $\Leg\F$ is described by the differential equation $F(p,q,w):=qw^3+pw^2+(c-q)w+1=0$, where $w=\frac{\mathrm{d}q}{\mathrm{d}p}$. The explicit computation of $K(\Leg\F)$ shows that it has the form $$K(\Leg\F)=\frac{\sum\limits_{i+j\leq6}\rho_{i}^{j}(c)p^iq^j}{\Delta(p,q)^2}\mathrm{d}p\wedge\mathrm{d}q,$$
      where $\Delta$ is the $w$-discriminant of $F$ and the $\rho_{i}^{j}$'s are polynomials in $c$ with $\rho_{1}^{5}(c)=4\neq0$; hence $K(\Leg\F)\not\equiv0.$

      \noindent Similarly, we verify that $\Leg\F$ can not be flat when $\F$ is given by one of the last six $1$-forms of Proposition~\ref{pro:v=3-mu=13}.
\end{itemize}
\end{proof}
\bigskip

\subsubsection{The singular locus contains a point of algebraic multiplicity $3$ and is not reduced to this point}

We begin by proving four lemmas.

\begin{lem}\label{lem:v=3-leg-plat}
{\sl Let $\F$ be a foliation of degree three on $\pp$, let $m$ be a singular point of $\F$ and let $\omega$ be a $1$-form describing~$\F$. Assume that this singularity is of algebraic multiplicity $3$ and that the $3$-web $\Leg\F$ is flat. Then
\begin{itemize}
  \item [--] either $\F$ is homogeneous;
  \item [--] or the $3$-jet of $\omega$ at $m$ is not saturated.
\end{itemize}
}
\end{lem}

\begin{rem}\label{rem:caractérisation-feuilletage-homogène}
Let us note that a foliation of degree $d$ on $\pp$ is homogeneous if and only if it has a singularity of maximal algebraic multiplicity ({\it i.e.} equal to $d$) and an invariant line not passing through this singularity.
\end{rem}

\begin{proof}
Let us choose a system of homogeneous coordinates $[x:y:z]\in\pp$ in which $m=[0:0:1]$. The condition $\nu(\F,m)=3$ assures that every $1$-form $\omega$ defining $\F$ in the affine chart $(x,y)$ is of type $\omega=\theta_3+C_3(x,y)(x\mathrm{d}y-y\mathrm{d}x)$, where $\theta_3$ (resp. $C_3$) is a homogeneous $1$-form (resp. a homogeneous polynomial) of degree $3$; the $1$-form $\theta_3$ represents the $3$-jet of $\omega$ at $(0,0).$

\noindent Let us assume that $\theta_3$ is saturated; we will prove that $\F$ is necessarily homogeneous. Let us denote by $\mathcal{H}$ the homogeneous foliation of degree three on $\pp$ defined by $\theta_3$; $\mathcal{H}$ is well defined thanks to the hypothesis on $\theta_3.$ Let us consider the family of homotheties $\varphi=\varphi_{\varepsilon}=(\varepsilon\,x,\varepsilon\hspace{0.1mm}y).$ We have
$$
\frac{1}{\varepsilon^4}\varphi^*\omega=\theta_3+\varepsilon\,C_3(x,y)(x\mathrm{d}y-y\mathrm{d}x)
$$
which tends to $\theta_3$ as $\varepsilon$ tends to $0$; it follows that $\mathcal{H}\in\overline{\mathcal{O}(\F)}.$ The $3$-web $\Leg\F$ is by hypothesis flat; it is therefore the same for the $3$-web $\Leg\mathcal{H}.$ The foliation $\mathcal{H}$ is then linearly conjugated to one of the eleven homogeneous foliations given by Theorem~5.1~of~\cite{BM17}. Thus, according to \cite[Table 1]{BM17}, $\mathcal{H}$ has at least one non-degenerate singularity $m_0$ satisfying $\mathrm{BB}(\mathcal{H},m_0)\not\in\{4,\frac{16}{3}\}.$ Let $(\F_{\varepsilon})_{\varepsilon\in\mathbb{C}}$ be the family of foliations defined by $\omega_{\varepsilon}=\theta_3+\varepsilon\,C_3(x,y)(x\mathrm{d}y-y\mathrm{d}x).$ From what precedes, for $\varepsilon\neq0$ the foliation $\F_{\varepsilon}$ belongs to $\mathcal{O}(\F)$ and for $\varepsilon=0$ we have $\F_{\varepsilon=0}=\mathcal{H}.$ The singularity $m_0$ of $\mathcal{H}$ is \og stable\fg; there is a family $(m_\varepsilon)_{\varepsilon\in\mathbb{C}}$ of non-degenerate singularities of $\F_{\varepsilon}$ such that $m_{\varepsilon=0}=m_0.$ The $\F_\varepsilon$'s being conjugated for $\varepsilon\neq 0$, $\mathrm{BB}(\F_{\varepsilon},m_\varepsilon)$ is locally constant; as a result $\mathrm{BB}(\F_{\varepsilon},m_\varepsilon)=\mathrm{BB}(\mathcal{H},m_0)$ for $\varepsilon$ small. In particular $\F$ has a non-degenerate singularity $m'$ verifying $\mathrm{BB}(\F,m')=\mathrm{BB}(\mathcal{H},m_0)$ so that $\mathrm{BB}(\F,m')\not\in\{4,\frac{16}{3}\}.$ According to~\cite[Lemma~6.7]{BM17} through the point $m'$ pass exactly two lines invariant by $\F$, of which at least one is necessarily distinct from $(mm')$; this implies, according to Remark~\ref{rem:caractérisation-feuilletage-homogène}, that $\F$ is homogeneous.
\end{proof}

\begin{lem}\label{lem:v=3-mu<13-jet-non-saturé}
{\sl Let $\mathcal{F}$ be a foliation of degree three on $\pp$ with a singular point $m$ of algebraic multiplicity~$3.$ Let $\omega$ be a $1$-form describing~$\F$. Assume that the singular locus of $\F$ is not reduced to $m$ and that the~$3$-jet of $\omega$ at $m$ is not saturated. Then up to isomorphism $\omega$ is of the following type
$$
y(a_0\,x^2+a_1xy+y^2)\mathrm{d}x+xy(b_0\,x+b_1y)\mathrm{d}y+x(x^2+c_1xy+c_2y^2)(x\mathrm{d}y-y\mathrm{d}x),
$$
where $a_0,a_1,b_0,b_1,c_1,c_2$ are complex numbers such that the degree of the associated foliation is $3.$
}
\end{lem}

\begin{proof}
The condition $\nu(\F,m)=3$ assures the existence of a system of homogeneous coordinates $[x:y:z]\in\pp$ in which $m=[0:0:1]$ and $\F$ is defined by a $1$-form $\omega$ of type $$\omega=A(x,y)\mathrm{d}x+B(x,y)\mathrm{d}y+C(x,y)(x\mathrm{d}y-y\mathrm{d}x),$$ where $A,$ $B$ and $C$ are homogeneous polynomials of degree $3.$ Since $J^3_{(0,0)}\omega$ is by hypothesis not saturated, we can write
\begin{align*}
&A(x,y)=(h_0\,x+h_1y)(a_0\,x^2+a_1xy+a_2y^2)&&\hspace{1.5mm} \text{and}  \hspace{1.5mm}&& B(x,y)=(h_0\,x+h_1y)(b_0\,x^2+b_1xy+b_2y^2).
\end{align*}
Let us write $C(x,y)=\sum_{i=0}^{3}c_i\hspace{0.1mm}x^{3-i}y^{i}$. The hypothesis $\Sing\F\neq\{m\}$ allows us to assume that the point $m'=[0:1:0]$ is singular of $\F,$ which amounts to assuming that $c_3=h_1b_2=0$. The foliation $\F$ being of degree three the product $h_1a_2$ is non-zero and as a result $b_2=0$; replacing $h_0=h_0'h_1,\,a_i=a_i'a_2,\,b_i=b_i'a_2,\,c_{\hspace{-0.4mm}j}=c_{\hspace{-0.4mm}j}'h_1a_2$, with $i\in\{0,1\}$ and $j\in\{0,1,2\}$, we can assume that $h_1=a_2=1.$ Thus $\omega$ writes
\begin{align*}
&\omega=(h_0\,x+y)\left((a_0\,x^2+a_1xy+y^2)\mathrm{d}x+x(b_0\,x+b_1y)\mathrm{d}y\right)+x(c_0\,x^2+c_1xy+c_2y^2)(x\mathrm{d}y-y\mathrm{d}x).
\end{align*}

\noindent The conjugation by the diffeomorphism~$\left(x,y-h_0\,x\right)$ allows us to cancel $h_0$; as a consequence
\begin{align*}
&\omega=y\left((a_0\,x^2+a_1xy+y^2)\mathrm{d}x+x(b_0\,x+b_1y)\mathrm{d}y\right)+x(c_0\,x^2+c_1xy+c_2y^2)(x\mathrm{d}y-y\mathrm{d}x).
\end{align*}
The equality $\deg\F=3$ implies that $c_0\neq0$. By conjugating $\omega$ by the homothety $\left(\frac{1}{c_0}x,\frac{1}{c_0}y\right),$ we reduce ourselves to $c_0=1$, that is, to the announced model.
\end{proof}

\begin{lem}\label{lem:nu(F,m)=d}
{\sl Let $\mathcal{F}$ be a foliation of degree $d\geq2$ on $\pp.$ If $m\in\Sing\F$ is such that $\nu(\F,m)=d$, then for any $m'\in\Sing\F\setminus\{m\}$ we have~$\nu(\F,m')\leq d-1.$
}
\end{lem}

\begin{proof}
We know (\emph{see} for instance \cite[page 158]{Fis01}) that if $s$ is a singularity of $\F$ and if $\mathrm{X}=A(\mathrm{u},\mathrm{v})\frac{\partial{}}{\partial{\mathrm{u}}}+B(\mathrm{u},\mathrm{v})\frac{\partial{}}{\partial{\mathrm{v}}}$ is a vector field defining the germ of $\F$ at $s$, then $\nu(\F,s)^{2}\leq\nu(A,s)\cdot\nu(B,s)\leq\mu(\F,s)$. We also know (\emph{see} \cite{Bru00}) that $\sum\limits_{s\in\mathrm{Sing}\F}\mu(\F,s)=d^{2}+d+1.$ Let us assume now that there is $m\in\Sing\F$ such that $\nu(\F,m)=d$; let $m'$ be a point of $\Sing\F\setminus\{m\}.$ It follows that
$$\nu(\F,m)^2+\nu(\F,m')^2\leq d^{2}+d+1\Longrightarrow\nu(\F,m')\leq\sqrt{d+1}\Longrightarrow\nu(\F,m')\leq d-1,\hspace{1mm}\text{because}\hspace{1mm} d\geq2.$$
\end{proof}

\begin{lem}\label{lem:v=3-mu<13-leg-plat}
{\sl Let $\mathcal{F}$ be a foliation of degree three on $\pp$ with a singular point $m$ of algebraic multiplicity~$3.$ Assume that the singular locus of $\F$ is not reduced to $m$ and that the $3$-web $\Leg\F$ is flat. Then any singularity $m'$ distinct from $m$ is non-degenerate, the line $(mm')$ is $\F$-invariant and
\begin{itemize}
  \item [--] either $\F$ is homogeneous;
  \item [--] or $\mathrm{CS}(\F,(mm'),m')\in\{1,3\}$ for any $m'\in\Sing\F\smallsetminus\{m\}.$
\end{itemize}
}
\end{lem}

\begin{proof}
Let $m'$ be a singular point of $\F$ distinct from $ m.$ According to Lemma~\ref{lem:nu(F,m)=d} the equalities $\deg\F=3$ and $\nu(\F,m)=3$ imply that $\nu(\F,m')\leq2.$ If the singularity $m'$ were degenerate, then, according to Proposition~\ref{pro:cas-nilpotent-noeud-ordre-2}, the $3$-web $\Leg\F$ would not be flat, which is impossible by hypothesis. Therefore $\mu(\F,m')=1.$

\noindent Since $\deg\F=3,$\, $\tau(\F,m)=3$\, and \,$\tau(\F,m')\geq1$, the line $(mm')$ is invariant by $\F$ (otherwise we would have~$3=\deg\F=\sum_{p\in (mm')}\Tang(\F,(mm'),p)\geq\tau(\F,m)+\tau(\F,m')\geq4$).

\noindent Let us assume that it is possible to choose $m'$ in such a way that $\mathrm{CS}(\F,(mm'),m')\not\in\{1,3\}$; we will show that $\F$ is necessarily homogeneous. The equality $\mu(\F,m')=1$ and the condition $\mathrm{CS}(\F,(mm'),m')\neq1$ imply that $\mathrm{BB}(\F,m')\neq4.$
\begin{itemize}
  \item [--] If $\mathrm{BB}(\F,m')\neq\frac{16}{3}$, then, according to \cite[Lemma~6.7]{BM17}, through the point~$m'$ passes a line invariant by $\F$ and distinct from the line $(mm')$, which implies, according to Remark~\ref{rem:caractérisation-feuilletage-homogène}, that $\F$ is homogeneous;
  \item [--] If $\mathrm{BB}(\F,m')=\frac{16}{3},$ then, according to \cite[Lemma~3.12]{Bed17}, through the singularity $m'$ passes a line $\ell$ invariant by $\F$ and such that $\mathrm{CS}(\F,\ell,m')=3$; as we have assumed that $\mathrm{CS}(\F,(mm'),m')\neq3,$ we deduce that $\ell\neq(mm')$, which implies (Remark~\ref{rem:caractérisation-feuilletage-homogène}) that $\F$ is homogeneous.
\end{itemize}
\end{proof}
\bigskip

\noindent We are now able to describe up to isomorphism the foliations of $\mathbf{FP}(3)$ whose singular locus contains a point of algebraic multiplicity $3$ and is not reduced to this point.
\begin{pro}\label{pro:overline-omega1-omega4-omega5}
{\sl Let $\mathcal{F}$ be a foliation of degree three on $\pp.$ Assume that $\F$ has a singularity of algebraic multiplicity~$3$ and that $\Sing\F$ is not reduced to this singularity. Assume moreover that the $3$-web $\Leg\F$ is flat. Then either $\F$ is homogeneous, or $\F$ is, up to the action of an automorphism of $\pp,$ defined by one of the following $1$-forms
\begin{itemize}
\item [\texttt{1. }] $\omegaoverline_{1}=y^{3}\mathrm{d}x+x^{3}(x\mathrm{d}y-y\mathrm{d}x)$;

\item [\texttt{2. }] $\omegaoverline_{4}=(x^{3}+y^{3})\mathrm{d}x+x^{3}(x\mathrm{d}y-y\mathrm{d}x)$;

\item [\texttt{3. }] $\omegaoverline_{5}=y^{2}(y\mathrm{d}x+2x\mathrm{d}y)+x^{3}(x\mathrm{d}y-y\mathrm{d}x)$.
\end{itemize}
}
\end{pro}

\begin{proof}
\noindent Let us assume that $\F$ is not homogeneous; we must show that up to linear conjugation $\F$ is described by one of the three $1$-forms $\omegaoverline_{1},\omegaoverline_{4},\omegaoverline_{5}$.

\noindent Let us denote by $m$ the singularity of $\F$ of algebraic multiplicity~$3.$ Let $\omega$ be a $1$-form describing $\F$ in an affine chart $(x,y)$ of $\pp.$ Since by hypothesis $\Leg\F$ is flat, it follows, according to Lemma~\ref{lem:v=3-leg-plat}, that the~$3$-jet of $\omega$ at $m$ is not saturated. By hypothesis we have $\Sing\F\neq\{m\}$. As a result, Lemma~\ref{lem:v=3-mu<13-jet-non-saturé} assures us that $\omega$ is, up to isomorphism, of the following type   $$
y(a_0\,x^2+a_1xy+y^2)\mathrm{d}x+xy(b_0\,x+b_1y)\mathrm{d}y+x(x^2+c_1xy+c_2y^2)(x\mathrm{d}y-y\mathrm{d}x),\quad a_i,b_i,c_{\hspace{-0.4mm}j}\in\mathbb{C}.
$$
In this situation, $m=[0:0:1]$, $m':=[0:1:0]\in\Sing\F$ and the line $(mm')=(x=0)$ is invariant by $\F$; moreover, a straightforward computation shows that $\mathrm{CS}(\F,(mm'),m')=1+b_1.$ Lemma~\ref{lem:v=3-mu<13-leg-plat} then implies that $b_1\in\{0,2\}.$

\noindent If $b_0\neq0$,\hspace{2mm} resp. $(b_0,b_1)=(0,2)$,\hspace{2mm} resp. $b_0=b_1=0,c_2\neq0$,\hspace{2mm} resp. $b_0=b_1=c_2=0,c_1\neq0$, then by conjugating $\omega$ by
\begin{small}
\begin{align*}
&\left(\dfrac{b_0^2\,x}{1-c_1b_0\,x},\dfrac{b_0^3y}{1-c_1b_0\,x}\right),&&
\text{resp.}\hspace{1mm}
\left(\frac{x}{1-\left(\dfrac{c_2}{2}\right)x},\frac{y}{1-\left(\dfrac{c_2}{2}\right)x}\right),&&
\text{resp.}\hspace{1mm}
\left(c_2^{-1}x,c_2^{-3/2}\,y\right),&&
\text{resp.}\hspace{1mm}
\left(c_1^{-2}x,c_1^{-3}y\right),
\end{align*}
\end{small}
\hspace{-1mm}we reduce ourselves to $(b_0,c_1)=(1,0)$,\hspace{2mm} resp. $c_2=0$,\hspace{2mm} resp. $c_2=1$,\hspace{2mm} resp. $c_1=1.$ Therefore, it suffices us to treat the following possibilities
\begin{align*}
& (b_0,b_1,c_1)=(1,0,0),&&\hspace{4mm} (b_0,b_1,c_1)=(1,2,0),&&\hspace{4mm} (b_0,b_1,c_2)=(0,2,0),&&\\
& (b_0,b_1,c_2)=(0,0,1),&&\hspace{4mm} (b_0,b_1,c_1,c_2)=(0,0,1,0),&&\hspace{4mm} (b_0,b_1,c_1,c_2)=(0,0,0,0).
\end{align*}

\noindent Let us place ourselves in the affine chart $(p,q)$ of $\pd$ associated to the line $\{py-qx=1\}\subset\pp$; the $3$-web $\Leg\F$ is described by the differential equation $$F(p,q,w):=pw^3+(a_1p+b_1q-c_2)w^2+(a_0p+b_0q-c_1)w-1=0,\qquad \text{with} \qquad w=\frac{\mathrm{d}q}{\mathrm{d}p}.$$

\noindent The explicit computation of $K(\Leg\F)$ shows that it has the form $$K(\Leg\F)=\frac{\sum\limits_{i+j\leq6}\rho_{i}^{j}p^iq^j}{\Delta(p,q)^2}\mathrm{d}p\wedge\mathrm{d}q,$$
where $\Delta$ is the $w$-discriminant of $F$ and the $\rho_{i}^{j}$'s are polynomials in the parameters $a_i,b_i,c_{\hspace{-0.4mm}j}.$
\vspace{2mm}

\textbf{\textit{1.}} If $(b_0,b_1,c_1)=(1,0,0)$, then the explicit computation of $K(\Leg\F)$ leads to $\rho_{0}^{5}=4c_2$ and
\begin{small}
$$
\rho_{2}^{1}=a_0^2(a_0+8a_1)c_2^3-(177a_0^2-60a_0-24a_0a_1-84a_0a_1^2+24a_1^2+24a_1^3)c_2^2+(108a_1^2-36a_1-81a_0)c_2+81,
$$
\end{small}
\hspace{-1mm}so that the system $\rho_{0}^{5}=\rho_{2}^{1}=0$ has no solutions. So this first case does not happen.
\vspace{2mm}

\textbf{\textit{2.}} When $(b_0,b_1,c_1)=(1,2,0)$, the explicit computation of $K(\Leg\F)$ gives us:
\begin{small}
\begin{align*}
\begin{array}{lll}
\vspace{1mm}
\rho_{0}^{6}=-4(6a_0-5a_1+4),\\
\vspace{1mm}
\rho_{1}^{5}=-4(12a_0^2+20a_0-10a_0a_1-3a_1^2),\\
\rho_{5}^{0}=32a_0^5c_2-8(a_1^2c_2+2a_1c_2+12)a_0^4+4(a_1^3c_2+4a_1^2+a_1-12)a_0^3+(4a_1^2-5a_1+30)a_0^2a_1^2-4a_0a_1^4;
\end{array}
\end{align*}
\end{small}
\hspace{-1mm}it is easy to see that the system $\rho_{0}^{6}=\rho_{1}^{5}=\rho_{5}^{0}=0$ has no solutions. So this second case is not possible.
\vspace{2mm}

\textbf{\textit{3.}} If $(b_0,b_1,c_2)=(0,2,0)$, then the explicit computation of $K(\Leg\F)$ shows that:
\begin{align*}
&\rho_{1}^{0}=24c_1^4,&&\rho_{1}^{4}=-256a_0^2,&&\rho_{0}^{4}=64(14a_1+3a_0c_1),
\end{align*}
so that $a_0=a_1=c_1=0.$ As a consequence, in this third case, $\F$ is given by $$\omegaoverline_{5}=y^2(y\mathrm{d}x+2x\mathrm{d}y)+x^3(x\mathrm{d}y-y\mathrm{d}x);$$
we verify by computation that its \textsc{Legendre} transform is flat.
\vspace{2mm}

\textbf{\textit{4.}} When $(b_0,b_1,c_2)=(0,0,1)$, the explicit computation of $K(\Leg\F)$ gives us:
\begin{Small}
\begin{align*}
\begin{array}{llll}
\vspace{1mm}
\rho_{5}^{0}=2a_0^2(4a_0-a_1^2)(2a_0^2-6a_0-a_0a_1c_1+2a_1^2),\\
\vspace{1mm}
\rho_{1}^{0}=2\left(c_1^2-4\right)\left((12a_0-a_1^2)c_1^2-(5a_0+18)a_1c_1+a_0^2-24a_0+14a_1^2+27\right),\\
\vspace{1mm}
\rho_{4}^{0}=(24a_0-5a_1^2)a_0^2a_1c_1^2-(60a_0^3-36a_0^2-11a_0^2a_1^2+45a_0a_1^2-8a_1^4)a_0c_1+2a_1(2a_0^4+24a_0^3-54a_0^2-5a_0^2a_1^2+30a_0a_1^2-4a_1^4),\\
\rho_{2}^{0}=(8a_0-a_1^2)a_1c_1^4-(64a_0^2+60a_0-7a_0a_1^2+7a_1^2)c_1^3+3(5a_0^2+52a_0-2a_1^2+27)a_1c_1^2-2a_1(a_0^2+162a_0-48a_1^2-27)\\
\hspace{6.5mm}-(a_0^3-177a_0^2-81a_0+84a_0a_1^2+108a_1^2+81)c_1;
\end{array}
\end{align*}
\end{Small}
\hspace{-1mm}the system $\rho_{1}^{0}=\rho_{2}^{0}=\rho_{4}^{0}=\rho_{5}^{0}=0$ is equivalent to $(a_0,a_1,c_1)\in\{(1,2,2),(1,-2,-2)\}$, as an explicit computation shows. If $(a_0,a_1,c_1)=(1,2,2)$, resp. $(a_0,a_1,c_1)=(1,-2,-2)$, then $\omega$ writes
\begin{align*}
&\omega=(x+y)^2(y\mathrm{d}x+x(x\mathrm{d}y-y\mathrm{d}x)),&&
\text{resp.}\hspace{1mm}
\omega=(x-y)^2(y\mathrm{d}x+x(x\mathrm{d}y-y\mathrm{d}x)),
\end{align*}
which contradicts the equality $\deg\F=3$.
\vspace{2mm}

\textbf{\textit{5.}} When $(b_0,b_1,c_1,c_2)=(0,0,1,0)$, the explicit computation of $K(\Leg\F)$ shows that:
\begin{small}
\begin{align*}
&\rho_{3}^{0}=-2a_0^2(6a_0+a_0a_1-2a_1^2)(4a_0-a_1^2),&& \rho_{0}^{0}=-81-60a_0+8a_0a_1+81a_1-7a_1^2-a_1^3,\\
&\rho_{1}^{0}=-4(a_1-3)(6a_0^2-9a_0a_1-a_0a_1^2+2a_1^3);
\end{align*}
\end{small}
\hspace{-1mm}the system $\rho_{0}^{0}=\rho_{1}^{0}=\rho_{3}^{0}=0$ is verified if and only if $(a_0,a_1)=(2,3),$ in which case $$\omega=(x+y)\big(y(y+2x)\mathrm{d}x+x^2(x\mathrm{d}y-y\mathrm{d}x)\big),$$ but this contradicts the equality $\deg\F=3.$
\vspace{2mm}

\textbf{\textit{6.}} If $(b_0,b_1,c_1,c_2)=(0,0,0,0),$ then $\omega=y(a_0\,x^2+a_1xy+y^2)\mathrm{d}x+x^3(x\mathrm{d}y-y\mathrm{d}x)$; the differential equation describing $\Leg\F$ writes
$$F(p,q,q')=p(q')^3+a_1p(q')^2+a_0pq'-1=0,\quad \text{with} \quad q'=\frac{\mathrm{d}q}{\mathrm{d}p}.$$
We study two eventualities according to whether $a_1$ is zero or not.
\vspace{2mm}

\textbf{\textit{6.1.}} When $a_1=0$ the explicit computation of $K(\Leg\F)$ gives us $$K(\Leg\F)=-\dfrac{48a_0^4\hspace{0.1mm}p}{(4a_0^3\hspace{0.1mm}p^2+27)^2}\mathrm{d}p\wedge\mathrm{d}q;$$
as a result $\Leg\F$ is flat if and only if $a_0=0$, in which case $\F$ is described by
$$\omegaoverline_{1}=y^3\mathrm{d}x+x^3(x\mathrm{d}y-y\mathrm{d}x).$$

\textbf{\textit{6.2.}} If $a_1\neq0$, then by conjugating $\omega$ by $\left(\alpha^2x,\alpha^3y\right)$, where $\alpha=\frac{1}{3}a_1$, we can assume that $a_1=3.$ In this case the explicit computation of $K(\Leg\F)$ shows that
\begin{align*}
K(\Leg\F)=-\dfrac{12\left(a_0-3\right)\left(a_0^2(4a_0-9)p+27(a_0-2)\right)}{\left(a_0^2(4a_0-9)p^2+54(a_0-2)p+27\right)^2}
\mathrm{d}p\wedge\mathrm{d}q;
\end{align*}
as a consequence $\Leg\F$ is flat if and only if $a_0=3$, in which case $$\omega=y(3x^2+3xy+y^2)\mathrm{d}x+x^3(x\mathrm{d}y-y\mathrm{d}x).$$ After replacing $\omega$ by $\varphi^*\omega,$ where $\varphi(x,y)=\left(x,-x-y\right),$ the foliation $\F$ is given in the affine coordinates $(x,y)$ by the $1$-form
$$\omegaoverline_{4}=(x^{3}+y^{3})\mathrm{d}x+x^{3}(x\mathrm{d}y-y\mathrm{d}x).$$
\end{proof}

\begin{rem}
The five foliations $\mathcal{F}_1,\ldots,\mathcal{F}_{5}$ have the following properties:
\vspace{0.2mm}
\begin{itemize}
\item [(i)]   $\#\Sing\F_2=1,$\, $\#\Sing\F_3=13$\, and \,$\#\Sing\F_j=2$ for $j=1,4,5$;
\item [(ii)]  $\F_j$ is convex if and only if $j\in\{1,3\}$;
\item [(iii)] $\F_j$ has a radial singularity of order $2$ if and only if $j\in\{1,3,4\}$;
\item [(iv)]  $\F_j$ admits a double inflection point if and only if $j\in\{2,4\}.$
\end{itemize}
\vspace{0.2mm}
The verifications of these properties are easy and left to the reader.
\end{rem}

\begin{rem}\label{rem:non-conjugaison-Hi-Fj}
The sixteen foliations $\mathcal{H}_1,$ $\ldots,\mathcal{H}_{11},$ $\mathcal{F}_1,$ $\ldots,\mathcal{F}_{5}$ are not linearly conjugated. Indeed, by construction, the $\F_j$'s are not homogeneous and are therefore not conjugated to the homogeneous foliations~$\mathcal{H}_{\hspace{0.2mm}i}.$ The $\mathcal{H}_{\hspace{0.2mm}i}$'s are not linearly conjugated (\cite[Theorem~5.1]{BM17}). Finally, the fact that the $\F_j$'s are not linearly conjugated follows from the properties (i),~(ii) and (iii) above.
\end{rem}

\noindent Theorem~\ref{thmalph:classification} follows from \cite[Theorems~5.1,~6.1]{BM17}, Propositions~\ref{pro:cas-nilpotent-noeud-ordre-2}, \ref{pro:overline-omega2}, \ref{pro:overline-omega1-omega4-omega5} and Remark~\ref{rem:non-conjugaison-Hi-Fj}.

\begin{proof}[\sl Proof of Corollary~\ref{coralph:class-convexe-3}]
According to \cite[Corollary~4.7]{BFM13} every convex foliation of degree three on $\pp$ has a flat \textsc{Legendre} transform and is therefore linearly conjugated to one of the sixteen foliations given by Theorem~\ref{thmalph:classification}. The statement then follows from the fact that the only convex foliations appearing in this theorem are~$\mathcal{H}_{1},\mathcal{H}_{\hspace{0.2mm}3},\F_1$ and $\F_3.$
\end{proof}

\section{Orbits under the action of $\mathrm{PGL}_3(\C)$}\label{sec:Orbites-Action-PGL-3}

\noindent In this section, we describe the irreducible components of $\mathbf{FP}(3)$. We start by determining the dimensions of the orbits
 $\mathcal{O}(\mathcal{H}_{\hspace{0.2mm}i}),\mathcal{O}(\F_j)$ under the action of $\mathrm{Aut}(\pp)=\mathrm{PGL}_3(\C)$. Next we classify up to isomorphism the foliations of~$\mathbf{F}(3)$ which realize the minimal dimension of the orbits in degree $3$. Finally, we study the closure of the orbits $\mathcal{O}(\mathcal{H}_{\hspace{0.2mm}i}),\mathcal{O}(\F_j)$ in $\mathbf{F}( 3)$ and we prove the Theorem~\ref{thmalph:12-Composantes irréductibles} describing
 the irreducible components of~$\mathbf{FP}(3).$

\Subsection{Isotropy groups and dimensions of the orbits $\mathcal{O}(\mathcal{H}_{\hspace{0.2mm}i})$ and $\mathcal{O}(\F_j)$}

\begin{defin}
Let $\F$ be a foliation on $\pp.$ The subgroup of $\mathrm{Aut}(\pp)$ (resp. $\mathrm{Aut}(\pd)$) which preserves $\F$ (resp. $\Leg\F$) is called the \textsl{isotropy group} of $\F$ (resp. $\Leg\F$) and is denoted by $\mathrm{Iso}(\F)$ (resp. $\mathrm{Iso}(\Leg\F)$); $\mathrm{Iso}(\F)$ and $\mathrm{Iso}(\Leg\F)$ are algebraic groups.
\end{defin}

\begin{rem}
Let $\F$ be a foliation on $\pp.$ If $[a:b:c]$ are the homogeneous coordinates in $\pd$ associated to the line $\{ax+by+cz=0\}\subset\pp,$ then
\begin{align*}
&\mathrm{Iso}(\Leg\F)=\Big\{[a:b:c]\cdot A^{-1}
\hspace{1mm}\big\vert\hspace{1mm}
A\in\mathrm{PGL}_3(\mathbb{C}),\hspace{1mm}[x:y:z]\cdot\transp{A}\in\mathrm{Iso}(\F)
\Big\}.
\end{align*}
More precisely, the isomorphism $\tau\hspace{1mm}\colon\mathrm{Aut}(\pp)\to\mathrm{Aut}(\pd)$ which, for $A$ in $\mathrm{PGL}_3(\mathbb{C}),$ sends $[x:y:z]\cdot\transp{A}$ into $[a:b:c]\cdot A^{-1}$ induces an isomorphism from $\mathrm{Iso}(\F)$ onto $\mathrm{Iso}(\Leg\F).$
\end{rem}

\noindent The following result is elementary and its proof is left to the reader.
\begin{pro}\label{pro:isotropies}
{\sl The groups $\mathrm{Iso}(\mathcal{H}_{\hspace{0.2mm}i})$ and $\mathrm{Iso}(\F_j)$ are given by
\smallskip
\begin{Small}
\begin{itemize}
\item [\texttt{1.}] $\mathrm{Iso}(\mathcal{H}_{1})=\Big\{
                    [\pm\,x:y:\alpha\,z],\hspace{1mm}
                    [\pm\,y:x:\alpha\,z]
                    \hspace{1mm}\big\vert\hspace{1mm}\alpha\in\mathbb{C}^*\Big\}$;
\medskip

\item [\texttt{2.}] $\mathrm{Iso}(\mathcal{H}_{\hspace{0.2mm}2})=\Big\{
                    [\pm\,x:y:\alpha\,z],\hspace{1mm}
                    [\pm\,y:x:\alpha\,z],\hspace{1mm}
                    [\pm\,\mathrm{i}\,x:y:\alpha\,z],\hspace{1mm}
                    [\pm\,\mathrm{i}y:x:\alpha\,z]
                    \hspace{1mm}\big\vert\hspace{1mm}\alpha\in\mathbb{C}^*\Big\}$;
\medskip

\item [\texttt{3.}] $\mathrm{Iso}(\mathcal{H}_{\hspace{0.2mm}3})=\Big\{[\,x:y:\alpha\,z],\hspace{1mm}[\,y:x:\alpha\,z]
                    \hspace{1mm}\big\vert\hspace{1mm}\alpha\in\mathbb{C}^*\Big\}$;
\medskip

\item [\texttt{4.}] $\mathrm{Iso}(\mathcal{H}_{\hspace{0.2mm}4})=\Big\{[\,x:y:\alpha\,z],\hspace{1mm}[\,y:x:\alpha\,z]
                    \hspace{1mm}\big\vert\hspace{1mm}\alpha\in\mathbb{C}^*\Big\}$;
\medskip

\item [\texttt{5.}] $\mathrm{Iso}(\mathcal{H}_{\hspace{0.2mm}5})=\Big\{[\,x:y:\alpha\,z]
                    \hspace{1mm}\big\vert\hspace{1mm}\alpha\in\mathbb{C}^*\Big\}$;
\medskip

\item [\texttt{6.}] $\mathrm{Iso}(\mathcal{H}_{\hspace{0.2mm}6})=\Big\{[\,x:y:\alpha\,z]
                    \hspace{1mm}\big\vert\hspace{1mm}\alpha\in\mathbb{C}^*\Big\}$;
\medskip

\item [\texttt{7.}] $\mathrm{Iso}(\mathcal{H}_{\hspace{0.2mm}7})=\Big\{
                    [\pm\,x:y:\alpha\,z]
                    \hspace{1mm}\big\vert\hspace{1mm}\alpha\in\mathbb{C}^*\Big\}$;
\medskip

\item [\texttt{8.}] $\mathrm{Iso}(\mathcal{H}_{\hspace{0.2mm}8})=\Big\{
                    [\,x:y:\alpha\,z],\hspace{1mm}
                    [4y-x:y:\alpha\,z]
                    \hspace{1mm}\big\vert\hspace{1mm}\alpha\in\mathbb{C}^*\Big\}$;
\medskip

\item [\texttt{9.}] $\mathrm{Iso}(\mathcal{H}_{\hspace{0.2mm}9})=\Big\{
                    [\,x:y:\alpha\,z],\hspace{1mm}
                    [\,x-y:x:\alpha\,z],\hspace{1mm}
                    [y:y-x:\alpha\,z]
                    \hspace{1mm}\big\vert\hspace{1mm}\alpha\in\mathbb{C}^*\Big\}$;
\medskip

\item [\texttt{10.}] $\mathrm{Iso}(\mathcal{H}_{\hspace{0.2mm}10})=\Big\{
                    [\,x:y:\alpha\,z],\hspace{1mm}
                    [-y:x:\alpha\,z]
                    \hspace{1mm}\big\vert\hspace{1mm}\alpha\in\mathbb{C}^*\Big\}$;
\medskip

\item [\texttt{11.}] $\mathrm{Iso}(\mathcal{H}_{\hspace{0.2mm}11})=\Big\{
                     [\,x:y:\alpha\,z],\hspace{1mm}
                     [y:x:\alpha\,z],\hspace{1mm}
                     [\xi^5\,x:x+\xi\,y:\alpha\,z],\hspace{1mm}
                     [\xi^{-5}\,x:x+\xi^{-1}\,y:\alpha\,z],\hspace{1mm}
                     [\xi^5\,y:y+\xi\,x:\alpha\,z],$
\item[]$
                     \hspace{2.13cm}
                     [\xi^{-5}\,y:y+\xi^{-1}\,x:\alpha\,z],\hspace{1mm}
                     [\xi^{5}\,x-y:x+\xi^{-1}\,y:\alpha\,z],\hspace{1mm}
                     [\xi^{-5}\,x-y:x+\xi\,y:\alpha\,z],$
\item[]$
                     \hspace{2.13cm}
                     [\xi^{5}\,x+\xi^{4}\,y:x:\alpha\,z],\hspace{1mm}
                     [\xi^{-5}\,x+\xi^{-4}\,y:x:\alpha\,z],\hspace{1mm}
                     [\xi^{5}\,y+\xi^{4}\,x:y:\alpha\,z],$
\item[]$
                     \hspace{2.13cm}
                     [\xi^{-5}\,y+\xi^{-4}\,x:y:\alpha\,z]
                     \hspace{1mm}\big\vert\hspace{1mm}\alpha\in\mathbb{C}^*\Big\}$ where $\mathrm{\xi}=\mathrm{e}^{\mathrm{i}\pi/6}$;
\medskip

\item [\texttt{12.}] $\mathrm{Iso}(\mathcal{F}_{1})=\Big\{
                     [\alpha^{2} x:\alpha^{3} y:z+\beta\,x]
                     \hspace{1mm}\big\vert\hspace{1mm} \alpha\in\mathbb{C}^*,\hspace{1mm}\beta\in\mathbb{C}\Big\}$;
\medskip

\item [\texttt{13.}] $\mathrm{Iso}(\mathcal{F}_{2})=\Big\{
                     [\alpha^{4} x:\alpha^{3} y:z+\beta\,x]
                     \hspace{1mm}\big\vert\hspace{1mm} \alpha\in\mathbb{C}^*,\hspace{1mm}\beta\in\mathbb{C}\Big\}$;
\medskip

\item [\texttt{14.}] $\mathrm{Iso}(\mathcal{F}_{3})=\Big\{
                     [\pm\,x:\pm\,y:z],\hspace{1mm}
                     [\pm\,y:\pm\,x:z],\hspace{1mm}
                     [\pm\,x:\pm\,z:y],\hspace{1mm}
                     [\pm\,z:\pm\,x:y],\hspace{1mm}
                     [\pm\,y:\pm\,z:x],\hspace{1mm}
                     [\pm\,z:\pm\,y:x]\Big\}$;
\medskip

\item [\texttt{15.}] $\mathrm{Iso}(\mathcal{F}_{4})=\Big\{
                     [\,x:y:z+\alpha x],\hspace{1mm}
                     [\,\mathrm{j}\,x:y:z+\alpha\,x],\hspace{1mm}
                     [\,\mathrm{j}^2x:y:z+\alpha\,x]
                     \hspace{1mm}\big\vert\hspace{1mm}\alpha\in\mathbb{C}\Big\}$ where $\mathrm{j}=\mathrm{e}^{2\mathrm{i}\pi/3}$;
\medskip

\item [\texttt{16.}] $\mathrm{Iso}(\mathcal{F}_{5})=\Big\{
                     [\alpha^{2} x:\alpha^{3} y:z]
                     \hspace{1mm}\big\vert\hspace{1mm}\alpha\in\mathbb{C}^*\Big\}$.
\end{itemize}
\end{Small}
\vspace{2mm}

\noindent In particular, the dimensions of the orbits $\mathcal{O}(\mathcal{H}_{\hspace{0.2mm}i})$ and $\mathcal{O}(\F_j)$ are the following
\begin{align*}
&\dim\mathcal{O}(\F_{1})=6,&&&\dim\mathcal{O}(\F_{2})=6,&&& \dim\mathcal{O}(\mathcal{H}_{\hspace{0.2mm}i})=7, i=1,\ldots,11,\\
&
\dim\mathcal{O}(\F_{4})=7,&&& \dim\mathcal{O}(\F_{5})=7,&&& \dim\mathcal{O}(\F_{3})=8.
\end{align*}
}
\end{pro}
\vspace{1mm}

\Subsection{Description  of degree three foliations $\F$ such that $\dim\mathcal{O}(\F)=6$}

\noindent Proposition~2.3~of~\cite{CDGBM10} asserts that if $\F$ is a foliation of degree  $d\geq2$ on $\pp,$ then the dimension of~$\mathcal{O}(\F)$ is at least $6,$ or equivalently, the dimension of~$\mathrm{Iso}(\F)$ is at most $2$.  Notice that these bounds are attained by the foliations  $\F_{1}^{(d)}$ and $\F_{2}^{(d)}$ defined in the affine chart $z=1$ respectively by the $1$-forms
\begin{align*}
&\omegaoverline_{1}^{(d)}\hspace{1mm}=y^{d}\mathrm{d}x+x^{d}(x\mathrm{d}y-y\mathrm{d}x)
&& \text{and} &&
\omegaoverline_{2}^{(d)}\hspace{1mm}=x^{d}\mathrm{d}x+y^{d}(x\mathrm{d}y-y\mathrm{d}x).
\end{align*}

\noindent Indeed, it is easy to check that
\begin{small}
\begin{align*}
&\left\{
\left(\frac{\alpha^{d-1}x}{1+\beta x},\frac{\alpha^{d}y}{1+\beta x}\right)
\hspace{1mm}\Big\vert\hspace{1mm}
\alpha\in\mathbb{C}^*,\hspace{1mm}\beta\in\mathbb{C}
\right\}
\subset\mathrm{Iso}(\F_{1}^{(d)})
&&\text{and}&&
\left\{
\left(\frac{\alpha^{d+1}x}{1+\beta x},\frac{\alpha^{d}y}{1+\beta x}\right)
\hspace{1mm}\Big\vert\hspace{1mm}
\alpha\in\mathbb{C}^*,\hspace{1mm}\beta\in\mathbb{C}
\right\}
\subset\mathrm{Iso}(\F_{2}^{(d)}),
\end{align*}
\end{small}
\hspace{-1mm}so that $\dim\mathrm{Iso}(\F_{i}^{(d)})\geq2,i=1,2,$ and so $\dim\mathrm{Iso}(\F_{i}^{(d)})=2.$
\begin{rem}
By construction, we have $\F_{1}^{(3)}=\F_1$ and $\F_{2}^{(3)}=\F_2.$
\end{rem}

\noindent D.~\textsc{Cerveau}, J.~\textsc{D\'eserti}, D.~\textsc{Garba Belko} and R.~\textsc{Meziani} have shown that up to isomorphism of $\pp$ the quadratic foliations $\F_{1}^{(2)}$ and $\F_{2}^{(2)}$ are the only foliations realizing the minimal dimension of the orbits in degree~$2$ (\cite[Proposition 2.7]{CDGBM10}). Corollary~\ref{coralph:dim-min} stated in the Introduction is a similar result in degree~$3$.

\begin{proof}[\sl Proof of Corollary~\ref{coralph:dim-min}]
Let $\F$ be a degree three foliation on  $\pp$  such that  $\dim\mathcal{O}(\F)=6.$ Since $\mathrm{Iso}(\Leg\F)$ is isomorphic to  $\mathrm{Iso}(\F),$ we have that $\dim\mathrm{Iso}(\Leg\F)=\dim\mathrm{Iso}(\F)=8-6=2.$ Let us fix $m\in\pd\smallsetminus\Delta(\Leg\F)$ and let  $\W_m$ be the germ of the $3$-web  $\Leg\F$ at $m.$
After \'{E}. \textsc{Cartan} \cite{Car08} the equality  $\dim\mathrm{Iso}(\Leg\F)=2$ implies that $\W_m$ is parallelizable and so flat. Since the curvature $\Leg\F$ is holomorphic on  $\pd\smallsetminus\Delta(\Leg\F),$ we deduce that $\Leg\F$ is flat. Therefore $\F$ is linearly conjugate to one of the $16$ foliations given by Theorem~\ref{thmalph:classification}. Proposition~\ref{pro:isotropies} and the hypothesis $\dim\mathcal{O}(\F)=6$ allows us to conclude.
\end{proof}

\Subsection{Closure of the orbits and irreducible components of $\mathbf{FP}(3)$}\label{subsec:Adhérence-Composantes-FP(3)}

\noindent We begin by studying the closure of the orbits $\mathcal{O}(\mathcal{H}_{\hspace{0.2mm}i})$ and $\mathcal{O}(\F_j)$ in $\mathbf{F}(3),$ then we prove Theorem~\ref{thmalph:12-Composantes irréductibles} describing the irreducible components of~$\mathbf{FP}(3).$

\noindent The following definition will be useful.
\begin{defin}[\cite{CDGBM10}]
Let $\F$ and $\F'$ be two foliations of~$\mathbf{F}(3).$ We say that $\F$ \textsl{degenerates} into $\F'$ if the closure $\overline{\mathcal{O}(\F)}$ (inside~$\mathbf{F}(3)$) of $\mathcal{O}(\F)$ contains~$\mathcal{O}(\F')$ and $\mathcal{O}(\mathcal{F})\not=\mathcal{O}(\mathcal{F}').$
\end{defin}

\begin{rems}\label{rems:orbt-cnvx-plt}
Let $\F$ and $\F'$ be two foliations such that $\F$ degenerates into $\F'$. Then
\begin{itemize}
\item [(i)]   $\dim\mathcal{O}(\F')<\dim\mathcal{O}(\F)$;

\item [(ii)]  if $\Leg\F$ is flat then $\Leg\F'$ is also flat;

\item [(iii)] $\deg\IinvF\leq\deg\mathrm{I}_{\mathcal{F}'}^{\mathrm{inv}}$, equivalently~$\deg\ItrF\geq\deg\mathrm{I}_{\mathcal{F}'}^{\hspace{0.2mm}\mathrm{tr}}$. In particular, if $\F$ is convex then $\F'$ is also convex.
\end{itemize}
\end{rems}

\noindent As we have already noted in the Introduction, D.~\textsc{Mar\'{\i}n} and J. \textsc{Pereira} have shown in~\cite{MP13} that the closure of the orbit $\mathcal{O}(\F_3)$ of $\F_3$ is an irreducible component of  $\mathbf{FP}(3).$ Assertion~\textbf{\textit{2.}} in proposition below gives a more precise description.

\begin{pro}\label{pro:adh-F1-F2-F3}
{\sl
\textbf{\textit{1.}} The orbits $\mathcal{O}(\F_1)$ and $\mathcal{O}(\F_2)$ are closed.

\noindent\textbf{\textit{2.}} $\overline{\mathcal{O}(\F_3)}=\mathcal{O}(\F_1)\cup\mathcal{O}(\mathcal{H}_{1})\cup\mathcal{O}(\mathcal{H}_{\hspace{0.2mm}3}) \cup\mathcal{O}(\F_3).$
}
\end{pro}

\begin{proof}
First assertion follows from Corollary~\ref{coralph:dim-min} and Remark~\ref{rems:orbt-cnvx-plt}~(i).

\noindent By Corollary~\ref{coralph:class-convexe-3} and Remark~\ref{rems:orbt-cnvx-plt}~(iii), $\F_3$ can degenerate only into  $\F_1,\mathcal{H}_{1}$ or $\mathcal{H}_{\hspace{0.2mm}3}.$
Let us show that this is the case.
Consider the family of homotheties $\varphi=\varphi_{\varepsilon}=\left(\frac{x}{\varepsilon},\frac{y}{\varepsilon}\right).$ We have that
\begin{align*}
-\varepsilon^4\varphi^*\omegaoverline_{3}=(y^3-\varepsilon^2y)\mathrm{d}x+(\varepsilon^2x-x^3)\mathrm{d}y
\end{align*}
tends to  $\omega_{1}$ as $\varepsilon$ goes to $0.$ Thus, the foliation $\F_{3}$ degenerates into $\mathcal{H}_{1}.$

\noindent In the affine chart $x=1$, $\F_1,$ resp. $\F_3,$ is given by
\begin{align*}
& \thetaoverline_1=\mathrm{d}y-y^3\mathrm{d}z, && \text{resp. } \thetaoverline_3=(y^3-y)\mathrm{d}z-(z^3-z)\mathrm{d}y;
\end{align*}
consider the family of automorphisms $\sigma=\left(\frac{y}{\varepsilon},2+6\varepsilon^{2}z\right).$ A direct computation shows that
\begin{align*}
-\frac{\varepsilon}{6}\sigma^*\thetaoverline_3=(1+11\varepsilon^{2}z+36\varepsilon^{4}z^{2}+36\varepsilon^{6}z^{3})\mathrm{d}y+(\varepsilon^{2}y-y^{3})\mathrm{d}z
\end{align*}
which tends to  $\thetaoverline_1$ as $\varepsilon$ tends to $0.$ Thus $\F_3$ degenerates into $\F_{1}.$

\noindent In homogeneous coordinates $\mathcal{H}_{\hspace{0.2mm}3}$, resp. $\F_3,$ is given by
\begin{align*}
&\Omega_3=z\,y^2(3x+y)\mathrm{d}x-z\,x^2(x+3y)\mathrm{d}y+xy(x^2-y^2)\mathrm{d}z,&&\\
\text{resp}.\hspace{1.5mm}
&\Omegaoverline_3=x^3(y\mathrm{d}z-z\mathrm{d}y)+y^3(z\mathrm{d}x-x\mathrm{d}z)+z^3(x\mathrm{d}y-y\mathrm{d}x);
\end{align*}
by putting $\psi=\left[x-y:2\varepsilon\,z-x-y:x+y\right]$ we obtain
\begin{align*}
\frac{1}{8\varepsilon}\psi^*\Omegaoverline_3=
z\,y(y-\varepsilon\,z)(3x+y-2\varepsilon\,z)\mathrm{d}x-z\,x(x-\varepsilon\,z)(x+3y-2\varepsilon\,z)\mathrm{d}y+xy(x^2-y^2)\mathrm{d}z
\end{align*}
which tends to $\Omega_3$ as $\varepsilon$ goes to $0.$ As a consequence $\F_3$ degenerates into $\mathcal{H}_{\hspace{0.2mm}3}.$
\end{proof}

\begin{rem}
By combining Assertion~\textbf{\textit{2.}} of Proposition~\ref{pro:adh-F1-F2-F3} and Corollary~\ref{coralph:class-convexe-3}, we deduce that the set of convex foliations of degree three on $\pp$ is exactly the closure $\overline{\mathcal{O}(\F_3)}$ of $\mathcal{O}(\F_3)$ and is therefore an irreducible closed subset of~$\mathbf{F}(3).$
\end{rem}

\noindent Next result is an immediate consequence of Corollary~\ref{coralph:dim-min} and Remark~\ref{rems:orbt-cnvx-plt}~(i).
\begin{cor}\label{cor:dim-o(F)=<7}
{\sl Let  $\F$ be an element of $\mathbf{F}(3)$ such that $\dim\mathcal{O}(\F)\leq7.$ Then
$$\overline{\mathcal{O}(\F)}\subset\mathcal{O}(\F)\cup\mathcal{O}(\F_1)\cup\mathcal{O}(\F_2).$$
}
\end{cor}

\noindent The following result provides a necessary condition for a degree three foliation on $\pp$ degenerates into the foliation $\F_1$.

\begin{pro}\label{pro:cond-nécess-dgnr-F1}
{\sl Let $\F$ be an element of $\mathbf{F}(3)$. If $\F$ degenerates into $\F_1$, then $\F$ possesses a non-degenerate singular point $m$ satisfying $\mathrm{BB}(\F,m)=4.$
}
\end{pro}

\begin{proof}
Assume that $\F$ degenerates into $\F_1$. Then there exists an analytic family $(\F_\varepsilon)$ of foliations defined by $1$-forms $\omega_\varepsilon$ such that $\F_\varepsilon\in\mathcal{O}(\F)$ for $\varepsilon\neq 0$ and  $\F_{\varepsilon=0}=\F_1$. The non-degenerate singular point $m_0$ of $\F_1$ is \og stable\fg, {\it i.e.} there is an analytic family $(m_\varepsilon)$ of non-degenerate singular points of $\F_\varepsilon$ such that $m_{\varepsilon=0}=m_0$. The $\F_\varepsilon$'s being conjugated to $\F$ for $\varepsilon\neq 0$, the foliation $\F$ admits a non-degenerate singular point $m$ such that
$$
\forall\hspace{1mm}\varepsilon\in\mathbb{C}^*,\hspace{1mm}\mathrm{BB}(\F_\varepsilon,m_\varepsilon)=\mathrm{BB}(\F,m).
$$
Since $\mu(\F_\varepsilon,m_\varepsilon)=1$ for every $\varepsilon$ in $\mathbb{C}$, the function  $\varepsilon\mapsto\mathrm{BB}(\F_\varepsilon,m_\varepsilon)$ is continuous, hence constant on $\mathbb{C}$. As a result
\begin{align*}
&&\mathrm{BB}(\F,m)=\mathrm{BB}(\F_{\varepsilon=0},m_{\varepsilon=0})=\mathrm{BB}(\F_1,m_0)=4.
\end{align*}
\end{proof}

\begin{cor}\label{cor:H2-H8-H11-F5}
{\sl The foliations $\mathcal{H}_{\hspace{0.2mm}2},\mathcal{H}_{\hspace{0.2mm}8},\mathcal{H}_{\hspace{0.2mm}11}$ and $\F_{5}$ do not degenerate into  $\F_1.$
}
\end{cor}

\noindent A sufficient condition for the degeneration of a degree three foliation into $\F_1$ is the following:

\begin{pro}\label{pro:cond-suff-dgnr-F1}
{\sl Let $\F$ be an element of $\mathbf{F}(3)$ such that $\F_1\not\in\mathcal{O}(\F).$ If $\F$ possesses a non-degenerate singular point $m$ satisfying $$\mathrm{BB}(\F,m)=4\qquad \text{and}\qquad \kappa(\F,m)=3,$$
then $\F$ degenerates into $\F_1.$
}
\end{pro}

\begin{proof}
Assume that $\F$ has a such singular point $m$. The equality  $\kappa(\F,m)=3$ assures the existence of a line $\ell_m$ through $m$ which is not invariant by $\F$ and such that $\Tang(\F,\ell_m,m)=3$. Taking an affine coordinate system $(x,y)$ such that $m=(0,0)$ and $\ell_m=(x=0)$, the foliation $\F$ is defined by a $1$-form $\omega$ of the following type
\begin{align*}
\hspace{1cm}& (*x+\beta y+*x^2+*xy+*y^2+*x^3+*x^2y+*xy^2+*y^3)\mathrm{d}x\\
\hspace{1cm}&\hspace{3mm}+(\alpha\,x+ry+*x^2+*xy+s\hspace{0.1mm}y^2+*x^3+*x^2y+*xy^2+\gamma\,y^3)\mathrm{d}y\\
\hspace{1cm}&\hspace{3mm}+(*x^3+*x^2y+*xy^2+*y^3)(x\mathrm{d}y-y\mathrm{d}x),&&
\text{with}\hspace{1mm} *,r,s,\alpha,\beta,\gamma\in\mathbb{C}.
\end{align*}
Along the line $x=0$ the $2$-form  $\omega\wedge\mathrm{d}x$ writes as
$(ry+s\hspace{0.1mm}y^2+\gamma\,y^3)\mathrm{d}y\wedge\mathrm{d}x$. The equality~$\Tang(\F,\ell_m,m)=3$ is equivalent to $r=s=0$ and $\gamma\neq0.$ The equalities $r=0$, $\mu(\F,m)=1$ and $\mathrm{BB}(\F,m)=4$ imply that $\beta=-\alpha\neq0.$ Thus $\omega$ writes as
\begin{align*}
\hspace{1cm}& (*x-\alpha y+*x^2+*xy+*y^2+*x^3+*x^2y+*xy^2+*y^3)\mathrm{d}x\\
\hspace{1cm}&\hspace{3mm}+(\alpha\,x+*x^2+*xy+*x^3+*x^2y+*xy^2+\gamma\,y^3)\mathrm{d}y\\
\hspace{1cm}&\hspace{3mm}+(*x^3+*x^2y+*xy^2+*y^3)(x\mathrm{d}y-y\mathrm{d}x),&&
\text{where } *\in\mathbb{C},\hspace{1mm}\alpha,\gamma\in\mathbb{C}^*.
\end{align*}
Put $\varphi=\left(\varepsilon^3\hspace{0.1mm}x,\varepsilon\hspace{0.1mm}y\right)$ and fix $(i,j)\in\mathbb{Z}^2_+\setminus\{(0,0)\}$. Notice that
\begin{itemize}
\item[\texttt{1. }] $\varphi^*(x^iy^j\mathrm{d}x)=\varepsilon^{3i+j+3}x^iy^j\mathrm{d}x$ is divisible by $\varepsilon^4$ and
                    $\frac{1}{\varepsilon^4}\varphi^*(x^iy^j\mathrm{d}x)$ tends to $0$ as $\varepsilon$ tends to $0$  except for
                    $(i,j)=(0,1)$;

\item[\texttt{2. }] $\varphi^*(x^iy^j\mathrm{d}y)=\varepsilon^{3i+j+1}x^iy^j\mathrm{d}y$ is divisible by $\varepsilon^4$ except for $(i,j)=(0,1)$ and
                    $(i,j)=(0,2).$ If $(i,j)\notin\{(0,1),(0,2),(0,3),(1,0)\},$ then the $1$-form $\frac{1}{\varepsilon^4}\varphi^*(x^iy^j\mathrm{d}y)$ tends to $0$ as $\varepsilon$ goes to $0.$
\end{itemize}
\noindent Therefore
\begin{align*}
\lim_{\varepsilon\to 0}\frac{1}{\varepsilon^4}\varphi^*\omega=\alpha(x\mathrm{d}y-y\mathrm{d}x)+\gamma\,y^3\mathrm{d}y.
\end{align*}
The foliation defined by $\alpha(x\mathrm{d}y-y\mathrm{d}x)+\gamma\,y^3\mathrm{d}y$ is conjugated to $\F_1$ because, as a straightforward computation shows, it is a convex foliation whose singular locus is formed of two points. As a result $\F$ degenerates into~$\F_1$.
\end{proof}

\begin{cor}\label{cor:H1-H3-H5-H7-F4}
{\sl
The foliations  $\mathcal{H}_{\hspace{0.2mm}1},\mathcal{H}_{\hspace{0.2mm}3},\mathcal{H}_{\hspace{0.2mm}5},\mathcal{H}_{\hspace{0.2mm}7}$ and $\F_{4}$ degenerate into $\F_1.$
}
\end{cor}

\noindent The converse of Proposition~\ref{pro:cond-suff-dgnr-F1} is false as the following example shows.
\begin{eg}\label{eg:sans-singularité-kappa=3}
Let $\F$ be the degree $3$ foliation on $\pp$ defined in the affine chart $z=1$ by
$$\omega=x\mathrm{d}y-y\mathrm{d}x+(y^2+y^3)\mathrm{d}y.$$
The singular locus of $\F$ consists of the two points  $m=[0:0:1]$ and $m'=[1:0:0]$; moreover
\begin{align*}
&\mu(\F,m)=1,&& \mathrm{BB}(\F,m)=4,&&\kappa(\F,m)=2,&& \mu(\F,m')>1.
\end{align*}
The foliation $\F$ degenerates into $\F_1$; indeed, putting
$\varphi=$\begin{small}$\left(\dfrac{1}{\varepsilon^3}x,\dfrac{1}{\varepsilon}y\right)$\end{small},
we have that
\begin{align*}
&\hspace{1.5cm}\lim_{\varepsilon\to 0}\varepsilon^4\varphi^*\omega=x\mathrm{d}y-y\mathrm{d}x+y^3\mathrm{d}y.
\end{align*}
\end{eg}
\vspace{1mm}

\noindent Next, we give a necessary condition for  a degree three foliation on $\pp$ degenerates into $\F_2$:
\begin{pro}
{\sl Let $\F$ be an element of $\mathbf{F}(3).$ If $\F$ degenerates into  $\F_2$, then $\deg\ItrF\geq2.$
}
\end{pro}

\begin{proof}
If $\F$ degenerates into $\F_2$ then $\deg\ItrF\geq\deg\mathrm{I}_{\F_2}^{\hspace{0.2mm}\mathrm{tr}}.$
A straightforward computation shows that  $\mathrm{I}_{\mathcal{F}_{2}}^{\hspace{0.2mm}\mathrm{tr}}=y^2$ so that $\deg\mathrm{I}_{\mathcal{F}_{2}}^{\hspace{0.2mm}\mathrm{tr}}=2$.
\end{proof}

\begin{cor}\label{cor:H5-H9}
{\sl
The foliations $\mathcal{H}_{\hspace{0.2mm}5}$ and $\mathcal{H}_{\hspace{0.2mm}9}$ do not degenerate into $\F_2$.
}
\end{cor}

\noindent A sufficient condition for that a degree three foliations on $\pp$ degenerates into $\F_2$ is the following:
\begin{pro}\label{pro:cond-suff-dgnr-F2}
{\sl Let $\F$ be an element of $\mathbf{F}(3)$ such that $\F_2\not\in\mathcal{O}(\F).$ If $\F$ possesses a double inflection point, then $\F$ degenerates into $\F_2$.}
\end{pro}

\begin{proof}
Assume that $\F$ possesses such a point $m$. We take an affine coordinate system $(x,y)$ such that $m=(0,0)$ is a double inflection point of $\F$ and $x=0$ is the tangent line to the leaf of $\F$ passing through $m$. Let $\omega$ be a $1$-form defining $\F$ in these coordinates. Since $\mathrm{T}_{\hspace{-0.4mm}m}\F=(x=0)$, $\omega$ has the following type
\begin{align*}
\hspace{0.5cm}&(\alpha+*x+*y+*x^2+*xy+*y^2+*x^3+*x^2y+*xy^2+*y^3)\mathrm{d}x\\
\hspace{0.5cm}&\hspace{3mm}+(*x+ry+*x^2+*xy+s\hspace{0.1mm}y^2+*x^3+*x^2y+*xy^2+\beta y^3)\mathrm{d}y\\
\hspace{0.5cm}&\hspace{3mm}+(*x^3+*x^2y+*xy^2+*y^3)(x\mathrm{d}y-y\mathrm{d}x),&&
\text{with}\hspace{1mm} *,r,s,\beta,\in\mathbb{C},\hspace{1mm}\alpha\in\mathbb{C}^*.
\end{align*}
Along the line $x=0$, the $2$-form $\omega\wedge\mathrm{d}x$ writes as $(ry+s\hspace{0.1mm}y^2+\beta\,y^3)\mathrm{d}y\wedge\mathrm{d}x$. The fact that $(0,0)$ is a double inflection point is equivalent to $r=s=0$ and $\beta\neq0.$ Thus $\omega$ writes as
\begin{align*}
\hspace{0.5cm}&(\alpha+*x+*y+*x^2+*xy+*y^2+*x^3+*x^2y+*xy^2+*y^3)\mathrm{d}x\\
\hspace{0.5cm}&\hspace{3mm}+(*x+*x^2+*xy+*x^3+*x^2y+*xy^2+\beta y^3)\mathrm{d}y\\
\hspace{0.5cm}&\hspace{3mm}+(*x^3+*x^2y+*xy^2+*y^3)(x\mathrm{d}y-y\mathrm{d}x),&&
\text{where }*\in\mathbb{C},\hspace{1mm}\alpha,\beta\in\mathbb{C}^*.
\end{align*}
We consider the following family of automorphisms $\varphi_\varepsilon=\varphi=(\varepsilon^4x,\varepsilon y).$ Fix $(i,j)\in\Z^2_+$ and notice that
\begin{itemize}
\item[\texttt{1. }] $\varphi^*(x^iy^j\mathrm{d}x)=\varepsilon^{4i+j+4}x^iy^j\mathrm{d}x$ is divisible by $\varepsilon^4$ and
                    $\frac{1}{\varepsilon^4}\varphi^*(x^iy^j\mathrm{d}x)$ tends to $0$ as $\varepsilon$ tends to $0$ except for  $i=j=0$;

\item[\texttt{2. }] $\varphi^*(x^iy^j\mathrm{d}y)=\varepsilon^{4i+j+1}x^iy^j\mathrm{d}y$ is divisible by $\varepsilon^4$ except for
                    $(i,j)\in\{(0,0),(0,1),(0,2)\}.$\\ If $(i,j)\notin\{(0,0),(0,1),(0,2),(0,3)\},$ then the $1$-form $\frac{1}{\varepsilon^4}\varphi^*(x^iy^j\mathrm{d}y)$ tends to $0$ as $\varepsilon$ goes to $0.$
\end{itemize}
We obtain that
\begin{align*}
\lim_{\varepsilon\to 0}\frac{1}{\varepsilon^4}\varphi^*\omega=\alpha\mathrm{d}x+\beta y^{3}\mathrm{d}y.
\end{align*}
Clearly $\alpha\mathrm{d}x+\beta y^{3}\mathrm{d}y$ defines a foliation which is conjugated to $\mathcal{F}_2$; as a result $\F$ degenerates into $\F_2.$
\end{proof}

\begin{cor}\label{cor:H2-H4-H6-H8-F4}
{\sl
The foliations $\mathcal{H}_{\hspace{0.2mm}2},\mathcal{H}_{\hspace{0.2mm}4},\mathcal{H}_{\hspace{0.2mm}6},\mathcal{H}_{\hspace{0.2mm}8}$ and $\F_{4}$ degenerate into $\F_2.$
}
\end{cor}

\begin{eg}[\textsc{Jouanolou}]
Consider the degree three foliation $\F_J$ on $\pp$ defined in the affine chart $z=1$~by $$\omega_J=(x^3y-1)\mathrm{d}x+(y^3-x^4)\mathrm{d}y;$$
this example is due to \textsc{Jouanolou} (\cite{Jou79}). Historically it is the first explicit example of foliation without invariant algebraic curves  (\cite{Jou79}); it is also a foliation without non-trivial minimal set (\cite{CdF01}).
The point $m=(0,0)$ is a double inflection point of  $\F_J$ because $\mathrm{T}_{\hspace{-0.4mm}m}\F_J=(x=0)\hspace{2mm} \text{and} \hspace{2mm}\omega_J\wedge\mathrm{d}x\Big|_{x=0}=y^3\mathrm{d}y\wedge\mathrm{d}x$; thus $\F_J$ degenerates into~$\F_2.$
\end{eg}

\noindent The converse of Proposition~\ref{pro:cond-suff-dgnr-F2} is false as the following example shows.
\begin{eg}\label{eg:sans-inflex-double}
Let $\F$ be the degree $3$ foliation on $\pp$ defined in the affine chart $z=1$ by
$$\omega=\mathrm{d}x+(y^2+y^3)\mathrm{d}y.$$
A straightforward computation shows that $\F$ has no double inflection point. This foliation degenerates into $\F_3$ in the following way. Putting $\varphi=$\begin{small}$\left(\dfrac{1}{\varepsilon^4}x,\dfrac{1}{\varepsilon}y\right)$\end{small}, we obtain that
\begin{align*}
&\hspace{1.5cm}\lim_{\varepsilon\to 0}\varepsilon^4\varphi^*\omega=\mathrm{d}x+y^3\mathrm{d}y.
\end{align*}
\end{eg}
\vspace{2mm}

\noindent
Theorem~\ref{thmalph:12-Composantes irréductibles} follows directly from Theorem~\ref{thmalph:classification}, Propositions~\ref{pro:isotropies}, \ref{pro:adh-F1-F2-F3} and Corollaries~\ref{coralph:class-convexe-3}, \ref{cor:dim-o(F)=<7}, \ref{cor:H2-H8-H11-F5}, \ref{cor:H1-H3-H5-H7-F4}, \ref{cor:H5-H9}, \ref{cor:H2-H4-H6-H8-F4}.
\smallskip

\begin{prob}\label{prob:dgnr-F1-F2}
{\sl
Give a criterion for deciding whether or not a degree three foliation on $\pp$ degenerates into $\F_1$ or $\F_2$.
}
\end{prob}

\noindent Thanks to Corollary~\ref{cor:dim-o(F)=<7}, an affirmative answer to this problem would allows us to decide whether or not an orbit of dimension $7$ in $\mathbf{F}(3)$ is closed.

%
%
%


\end{document}